\documentclass[10pt,reqno]{amsart}

\usepackage{appendix}
\usepackage[
  a4paper,
  top=1.5cm,
  bottom=1.5cm,
  left=2cm,
  right=2cm,
  heightrounded,
  marginparwidth=3.5cm,
  marginparsep=.2cm,
  centering
]{geometry}

\usepackage{xcolor}
\usepackage[utf8]{inputenc}
\usepackage[
  colorlinks=true,
  hyperindex,
  pagebackref,
  linktocpage=true
]{hyperref}
\usepackage{aliascnt}
\usepackage[nameinlink,noabbrev]{cleveref}
\usepackage{amsfonts}
\usepackage{amssymb}
\usepackage{mathrsfs}
\usepackage{mathtools,enumitem}
\usepackage{microtype,etoolbox}
\hypersetup{
  colorlinks,
  linkcolor={violet!60!black},
  citecolor={cyan!60!black},
  urlcolor={orange!60!black}
}

\theoremstyle{plain}
\newtheorem{theorem}{Theorem}[section]
\newaliascnt{proposition}{theorem}
\newtheorem{proposition}[proposition]{Proposition}
\aliascntresetthe{proposition}
\newaliascnt{lemma}{theorem}
\newtheorem{lemma}[lemma]{Lemma}
\aliascntresetthe{lemma}
\newaliascnt{corollary}{theorem}

\aliascntresetthe{corollary}
\newaliascnt{conjecture}{theorem}

\aliascntresetthe{conjecture}

\theoremstyle{definition}
\newaliascnt{definition}{theorem}

\aliascntresetthe{definition}
\newaliascnt{example}{theorem}

\aliascntresetthe{example}
\newaliascnt{remark}{theorem}
\newtheorem{remark}[remark]{Remark}
\aliascntresetthe{remark}
\newaliascnt{convention}{theorem}

\aliascntresetthe{convention}
\newaliascnt{hypothesis}{theorem}

\aliascntresetthe{hypothesis}
\newaliascnt{assumption}{theorem}

\aliascntresetthe{assumption}

\numberwithin{equation}{section}
\newcommand{\Qp}{\mathbf Q_p}
\newcommand{\Zp}{\mathbf Z_p}

\newcommand{\A}{\mathbf A}
\newcommand{\Gm}{\mathbf G_m}
\newcommand{\Spa}{\operatorname{Spa}}

\newcommand{\Spec}{\operatorname{Spec}}
\newcommand{\Lie}{\operatorname{Lie}}
\newcommand{\Int}{\operatorname{Int}}
\newcommand{\Perf}{\operatorname{Perf}}
\newcommand{\Hom}{\operatorname{Hom}}
\newcommand{\rk}{\operatorname{rk}}
\newcommand{\htg}{\operatorname{ht}}
\newcommand{\coker}{\operatorname{coker}}
\newcommand{\length}{\operatorname{length}}
\newcommand{\id}{\operatorname{id}}
\newcommand{\GL}{\operatorname{GL}}

\newcommand{\pr}{\operatorname{pr}}
\newcommand{\disc}{\operatorname{disc}}
\newcommand{\Tr}{\operatorname{Tr}}
\newcommand{\diag}{\operatorname{diag}}
\newcommand{\M}{\mathcal M}
\newcommand{\FF}{\mathrm{FF}}
\newcommand{\HT}{\mathrm{HT}}
\newcommand{\dR}{\mathrm{dR}}
\newcommand{\an}{\mathrm{an}}
\newcommand{\Berk}{\mathrm{Berk}}

\makeatletter
\renewcommand\normalsize{%
  \@setfontsize\normalsize{9.5}{11.2}%
  \abovedisplayskip 3.5pt plus .5pt minus .5pt
  \belowdisplayskip \abovedisplayskip
  \abovedisplayshortskip 1.5pt plus .5pt
  \belowdisplayshortskip 2.5pt plus .5pt minus .5pt
}
\newcommand{\tightenspace}[3]{%
  \patchcmd{#1}{#2}{#3}{}{\errmessage{Spacing patch failed}}}
\tightenspace{\@maketitle}{\topskip42\p@}{\topskip20\p@}
\tightenspace{\@maketitle}{\dimen@34\p@}{\dimen@18\p@}
\tightenspace{\@setauthors}{\@topsep30\p@}{\@topsep20\p@}
\tightenspace{\@setabstracta}{\skip@20\p@}{\skip@15\p@}
\tightenspace{\@starttoc}{\linespacing\@plus\linespacing}{6\p@\@plus\p@}
\tightenspace{\@starttoc}{.5\linespacing}{3\p@}
\tightenspace{\@starttoc}{32\p@\@plus14\p@}{10\p@\@plus\p@}
\renewcommand{\section}{\@startsection{section}{1}%
  \z@{7pt plus 1pt minus 1pt}{3pt}%
  {\normalfont\scshape\centering}}
\renewcommand{\subsection}{\@startsection{subsection}{2}%
  \z@{4pt plus 1pt minus 1pt}{-.4em}%
  {\normalfont\bfseries}}
\renewcommand{\subsubsection}{\@startsection{subsubsection}{3}%
  \z@{4pt plus 1pt minus 1pt}{-.4em}%
  {\normalfont\itshape}}
\def\thm@space@setup{%
  \thm@preskip=3.5pt plus .5pt minus .5pt
  \thm@postskip=\thm@preskip}
\expandafter\tightenspace\csname\string\proof\endcsname
  {\topsep6\p@\@plus6\p@}{\topsep3\p@\@plus.5\p@}
\makeatother
\normalsize
\setlist{topsep=3pt,itemsep=0pt,parsep=0pt,partopsep=0pt}
\begin{document}

\title{Steinness of the Basic $\mathrm{GL}_n$ Local Shimura Tower in Odd Rank}
\author{Jiawei Yang}
\address{Xiamen University, 422 South Siming Road, Siming District,
Xiamen, Fujian 361005, China}
\email{yangjw@stu.xmu.edu.cn}
\date{}

\begin{abstract}
We prove Steinness over $\breve{\Qp}$ for the basic local Shimura tower attached to $G=\mathrm{GL}_n,\;\mu=(1,1,0^{n-2}),$ 
with basic Newton slope $2/n$, for every odd $n\ge3$, at every finite level and for every prime $p$. The proof proceeds by constructing global analytic 
functions from the crystalline period map, proving compactness of their sublevel sets via a perfectoid normalization of Tate 
lattices, and showing that the resulting function map to affine space is finite. The key geometric inputs are the Fargues--Fontaine 
realization of the universal cover, a normalized determinant on the rank-two locus, and an integral PEL realization over
$W(\overline{\mathbf F}_p)$ used to produce bounded global Hodge generators. The integral constructions include the prime $p=2$.
\end{abstract}

\hypersetup{
 pdftitle={Steinness of the Basic GLn Local Shimura Tower in Odd Rank},
 pdfauthor={Jiawei Yang},
 pdfsubject={Local Shimura varieties and Stein exhaustion},
 pdfkeywords={local Shimura varieties, Rapoport-Zink spaces, Stein,
 perfectoid spaces, Fargues-Fontaine curve}
}
\maketitle
\tableofcontents

\section{Introduction}

A conjecture of Hansen--Scholze predicts that local Shimura varieties are Stein. It is formulated explicitly in
\cite[Conjecture~1.10]{HanSC}, following discussions between Hansen and Scholze. The known cases listed there include the
Lubin--Tate and Drinfeld towers, $\mu$-ordinary local Shimura data, and some Hodge--Newton reducible cases obtained from them.
Hansen singles out the basic $\mathrm{GL}_5$ datum with cocharacter $(1,1,0,0,0)$ as a difficult test case in
\cite[\S1.3]{HanSC} and returns to it in his 2023 lecture \cite{Han23}.

One motivation for the conjecture is its relation to cohomological vanishing in the local Langlands program.
As explained in \cite[\S1.3]{HanSC}, Steinness and Artin vanishing would give vanishing below the middle degree for the
supercuspidal cohomology considered there. Combined with duality when the associated Fargues--Scholze parameter is
supercuspidal, this would remove the global uniformization hypothesis in \cite[Theorem~1.1]{HanSC}.
More recently, Koshikawa and Shin observed that the Stein conjecture implies the lower bound in their conjectural
cohomological range for local shtuka spaces when $\mu$ is minuscule and $b$ is basic \cite[Remark~6.3.2]{KS26}.
These connections make Steinness relevant both to the global analytic geometry of local Shimura varieties and to the
cohomological degrees in which their representations occur.

We prove Steinness for the corresponding family in odd rank at every finite level and 
for every prime, including $p=2$. More strongly, each fixed-height hyperspecial component admits a finite morphism to an analytic 
affine space. The proof combines the Scholze--Weinstein description of infinite-level Rapoport--Zink spaces \cite{SW13,SW20} with 
integral Hodge sections and a uniform estimate for Tate lattices.

Let $n=2\ell+1\ge3$ be odd, let $p$ be a prime, put $k=\overline{\mathbb F}_p$ and $K=W(k)[1/p]=\breve{\mathbb Q}_p$, and consider
\[
G=\mathrm{GL}_n,\qquad \mu=(1,1,0^{n-2}),\qquad \nu_b=(2/n,\ldots,2/n),
\]
with $b$ basic. For a compact open subgroup $U\subset\mathrm{GL}_n(\mathbb Q_p)$, let $M_U$ be the corresponding local Shimura 
variety over $K$. Write $U_0=\mathrm{GL}_n(\mathbb Z_p)$ and denote by $M_{U_0}^{(h)}$ the quasi-isogeny locus of height $h$,
with the convention fixed in Section~\ref{subsec:framing}.

For an analytic domain $V\subset Y$, write $\Int_Y(V)$ for its interior in the Berkovich topology of $Y$; the subscript is omitted when the ambient space is clear.

\begin{theorem}\label{thm:main}
For every prime $p$, every odd $n\ge3$, and every compact open subgroup $U\subset\mathrm{GL}_n(\mathbb Q_p)$, the space $M_U$ admits an admissible exhaustion by affinoid domains
\[
M_U=\bigcup_{m\ge0}X_m,\qquad X_m\subset\operatorname{Int}_{M_U}(X_{m+1}),
\]
such that $X_m$ is a Weierstrass domain in $X_{m+1}$. In particular, each restriction map $\mathcal O(X_{m+1})\to\mathcal O(X_m)$ 
has dense image, and $M_U$ is Stein. The same conclusion holds for each locus of fixed height.
\end{theorem}

\begin{theorem}\label{thm:finiteintro}
For every prime $p$, every odd $n\ge3$, and every height $h$, there are finitely many global analytic functions $F_1,\ldots,F_d$ on $X=M_{U_0}^{(h)}$ 
for which
\[
\Psi=(F_1,\ldots,F_d): X\longrightarrow(\mathbb A_K^d)^{\mathrm{an}}
\]
is finite.
\end{theorem}

\smallskip\noindent\textit{Outline of the proof.} We use the moduli problem of Section~\ref{subsec:framing}, 
whose framing $p$-divisible group $\mathbb H$ has height $n$ and dimension $n-2$. The crystalline period map has Pl\"ucker 
coordinates $\ell_a\in\Gamma(X,\lambda^{-1})$, where $\lambda=\det\omega_{\mathcal H}$, and we put 
$H_{\mathrm{Pl}}=\max_a\|\ell_a\|$ using the integral model metric. An integral realization in a split unitary PEL moduli space,
together with \cite{VW13,Lan21}, provides Hodge generators. After normalizing the determinant of the crystalline framing, 
we obtain sections $\sigma_i\in\Gamma(X,\lambda^w)$ for an even $w\ge2$, with $\max_i\|\sigma_i(x)\|=1$. Consequently,
\[
F_{iab}=\sigma_i\ell_a^{\,w-1}\ell_b, \qquad \max_{i,a,b}|F_{iab}(x)|=H_{\mathrm{Pl}}(x)^w.
\]
The ratios $F_{iab}/F_{iaa}=\ell_b/\ell_a$ recover the local coordinates of the \'etale period map, so the resulting function map has discrete geometric fibres.

The main difficulty is compactness of Hodge sublevels. Fix a complete algebraically closed extension $C/K$. The universal cover 
of $\mathbb H$ is described by sections of the stable Fargues--Fontaine bundle $\mathcal B=\mathcal O((n-2)/n)$, with quasi-logarithm 
given by evaluation at the untilt divisor. Stability of the bundle implies that every nonzero global section has nonzero evaluation, 
while any $\ell+1$ global sections linearly independent over $\Qp$ have evaluations of rank at least two. These estimates yield a uniform lower bound for selected two-by-two minors 
on a compact set of normalized frames.

A second invariant is obtained on the locus of rank at most two by dividing the determinant of $\mathcal O^n\to\mathcal B$ by 
the $(n-2)$-th power of the section defining the untilt divisor. Its nonvanishing locus identifies actual infinite-level objects, and 
its valuation is constant above each hyperspecial height component, using \cite{CKV15,CKVcorr}. Together with the  
bound of the selected minors, it controls the elementary divisors of an actual Tate lattice relative to its normalization in a fixed affinoid perfectoid 
ball. The resulting bound is uniform in the deformation, the period point, and complete algebraically closed extensions of $C$.

Jacobi's complementary-minor identity and the integral Hodge--Tate map bound the quasi-logarithmic minors by 
$p^hH_{\mathrm{Pl}}$. The lattice estimate therefore places all points of a fixed Hodge sublevel in the finite-level image of 
a single quasi-compact perfectoid space. Continuity of $H_{\mathrm{Pl}}$ then gives that the sublevel is compact in the Berkovich 
topology. The norm identity above makes $\Psi_C$ topologically proper; partial properness of the source supplies the absence of 
boundary. Its zero-dimensional fibres imply finiteness, which descends to $K$ \cite{CT21,Zav25}.

The inverse images of expanding closed polydiscs under $\Psi$ form a Weierstrass exhaustion. Finite \'etale pullback and a diagonal exhaustion of the height loci give the assertion at every finite level.

\smallskip\noindent\textit{Organization.} Section~\ref{sec:hodge} constructs the global functions. Sections~\ref{sec:universalcover}--\ref{sec:uniform} establish compact normalization, the determinant invariant, and the uniform lattice bound. Section~\ref{sec:stein} proves compactness, finiteness, and Steinness. Appendix~\ref{app:PEL} gives the integral PEL realization and Hodge generators, including at $p=2$.

\section{Hodge bundles and global functions}\label{sec:hodge}

\subsection{The framing group and its Hodge bundles}\label{subsec:framing}
Let $\sigma$ denote the Witt-vector Frobenius on $K$.  Choose a representative $b\in \mathrm{GL}_n(K)$ of the basic $\sigma$-conjugacy class.  Write $D_b=(K^n,F=b\sigma)$ for the corresponding $F$-isocrystal; it is isoclinic of $F$-slope $2/n$.

Choose a full $W(k)$-lattice $M_0\subset D_b$ such that
\[
pM_0\subset F(M_0)\subset M_0.
\]
Equivalently, $M_0$ is stable under both $F$ and $V:=pF^{-1}$.  These operators satisfy $FV=VF=p$, so $M_0$ is a Dieudonn\'e 
lattice. By covariant Dieudonn\'e theory \cite[Theorem~4.33]{CO09}, there is a $p$-divisible group $\mathbb H/k$ with integral
Dieudonn\'e module $M_0$ and rational isocrystal $D_b$. The height is $n$. 
Since $v_p(\det F)=2$ and $V=pF^{-1}$, one has $v_p(\det V)=n-2.$ In the Chai--Oort normalization, $\Lie(\mathbb H)$ is the 
cokernel of $V$, hence $\dim \mathbb H=n-2.$ We fix this integral representative and call it the framing group.

For a $p$-nilpotent $W(k)$-scheme $S$, put $\bar S=S\times_{\Spec W(k)}\Spec k.$ The Rapoport--Zink functor parametrizes pairs $(\mathcal H,\rho)$ with $\rho:\mathbb H_{\bar S}\dashrightarrow \mathcal H_{\bar S}$ a quasi-isogeny.  In the EL convention of \cite[\S6.5]{SW13}, the dimension agrees with the multiplicity of weight $0$ in $\mu$, namely $n-2$.

Let $\mathcal B:=\mathcal E(\mathbb H)$ be the Fargues--Fontaine bundle attached to $\mathbb H$ in the Scholze--Weinstein convention.  It has rank $n$ and degree $n-2$, and since $\mathbb H$ is basic it is isoclinic.  Hence
$\mathcal B\simeq \mathcal O((n-2)/n).$ Since $n$ is odd, $\gcd(n,n-2)=1$, so $\mathcal B$ is stable.

We write $X=M_{U_0}^{(h)},\; U_0=\mathrm{GL}_n(\Zp),$ for the locus on which the quasi-isogeny $\rho$ has height $h$.

Let $\mathfrak M$ be the formal Rapoport--Zink space for $\mathbb H$, with universal $p$-divisible group $\mathcal H$ and universal 
quasi-isogeny $\rho:\mathbb H_{\bar S}\dashrightarrow \mathcal H_{\bar S}.$ Let $D_{\mathcal H}$ be the de Rham realization of the 
covariant Dieudonn\'e crystal of $\mathcal H$.  Its Hodge sequence is
\[
0\longrightarrow \Omega:=\omega_{\mathcal H^\vee}
\longrightarrow D_{\mathcal H}
\longrightarrow \Lie\mathcal H
\longrightarrow0.
\tag{2.1}
\]
Set $D:=D_{\mathcal H}^\vee,\; L:=\omega_{\mathcal H}\subset D,\; \lambda:=\det L.$ Then $\rk D=n,\; \rk L=n-2,\; \rk \Omega=2,$ and $L$ and $\Omega$ are mutual annihilators under the natural perfect pairing.

Let $M(\mathbb H)$ denote the Scholze--Weinstein covariant Dieudonn\'e lattice and put $N:=\bigl(M(\mathbb H)[1/p]\bigr)^\vee.$ The crystalline realization of $\rho$ gives on the generic fibre
\[
r_\rho:N^\vee\otimes_K\mathcal O_X\xrightarrow{\sim}D_{\mathcal H,\eta}.
\]
Dualizing gives the crystalline framing
\[
\alpha:=r_\rho^\vee:D_\eta\xrightarrow{\sim}N\otimes_K\mathcal O_X.
\tag{2.2}
\]
We also write $\beta=r_\rho^{-1}=(\alpha^\vee)^{-1}$.  At infinite level we reserve 
$\gamma:\underline{\Zp}^{\,n}\xrightarrow{\sim}T_p\mathcal H$ for the integral Tate frame. Here 
$\underline{\Zp}:=\varprojlim_{m\ge1}\underline{\mathbf Z/p^m\mathbf Z}$ is the constant $p$-adic sheaf on the pro-\'etale site, 
with constant finite sheaves on the right. Thus $\gamma$ is a trivialization of the Tate-module local system of rank $n$.

The Hodge subbundle $L\subset D$, transported through $\alpha$, defines the dual form of the crystalline period map
\[
\pi:X\longrightarrow \operatorname{Gr}(n-2,N),\qquad x\longmapsto \alpha_x(L_x),
\]
which is \'etale.  Choose a $W(k)$-basis of the dual reference lattice $M(\mathbb H)^\vee\subset N$.  The $\binom n2$ Pl\"ucker coordinates pull back to sections $\ell_a\in\Gamma(X,\lambda^{-1}),\; 1\le a\le\binom n2.$ The integral model of $\lambda$ gives a continuous model metric, and we put
\[
H_{\mathrm{Pl}}(x):=\max_{1\le a\le\binom n2}\|\ell_a(x)\|>0.
\tag{2.3}
\]
In an integral local basis of $L$, this is the maximum absolute value of the coefficients of $\bigwedge^{n-2}\alpha|_L$.

\subsection{The determinant model norm}\label{subsec:detnorm}
We normalize the absolute value by $|p|=p^{-1}$.

\begin{lemma}\label{lem:unitnorm}
Let $F$ be a complete nontrivially valued nonarchimedean field, and let $\mathbb D^q=\{(T_1,\ldots,T_q):|T_i|<1\}$ be the open unit polydisc.  If $f\in\mathcal O(\mathbb D^q)^\times$ is an analytic unit, then $|f(x)|=|f(0)|$ for every Berkovich point $x\in\mathbb D^q$.
\end{lemma}

\begin{proof}
Fix $0<r<1$ and restrict $f$ and $f^{-1}$ to the closed polydisc $\mathbb D_r^q:=\{x\in\mathbb D^q:\max_{1\le i\le q}|T_i(x)|\le r\}$ of radius $r$.  After a complete scalar extension, if necessary, we may assume that $r$ lies in the value group.  Rescaling the coordinates identifies the function algebra with a Tate algebra.  By the unit criterion for Tate algebras \cite[\S1.2, Corollary~4]{Bos08}, the constant term strictly dominates the nonconstant terms in the Gauss norm, so $\|f-f(0)\|_r<|f(0)|.$ Hence, for every $x\in\mathbb D_r^q$, $|f(x)-f(0)|<|f(0)|,$ and the ultrametric inequality gives $|f(x)|=|f(0)|$.

For $x\in\mathbb D_r^q$, let $\kappa(x)$ be its residue field and let $\widehat{\kappa(x)}$ be the completion for the absolute value defined by $x$.  Analytic functions evaluated at $x$ take values in this completed residue field.  To descend from the scalar extension, choose a complete valued extension $E/\widehat{\kappa(x)}$ containing an element of absolute value $r$.  The induced $E$-valued point $x_E$ has the same absolute values on analytic functions defined over $F$.  The equality over $E$ therefore gives the equality at $x$.  Every point of the open unit polydisc lies in some $\mathbb D_r^q$ with $r<1$.
\end{proof}

\begin{lemma}\label{lem:heightdet}
Let $f:G_1\to G_2$ be an isogeny of $p$-divisible groups over $k$, and put $a:=\htg(f)=\log_p\rk(\ker f).$ Let $M(G_i)$ denote the integral covariant Dieudonn\'e modules over $W(k)$.  Then
\[
\length_{W(k)}\coker M(f)=a,
\qquad
v_p(\det M(f))=a.
\]
For the quasi-isogeny $\rho_z:\mathbb H\dashrightarrow \mathcal H_z$ at a closed point $z$ of the height-$h$ locus in the special fibre of $\mathfrak M$, one has 
$v_p\bigl(\det M(\rho_z)\bigr)=h,$ where the determinant is computed in the chosen integral lattices.
\end{lemma}

\begin{proof}
Exactness of covariant Dieudonn\'e theory identifies the cokernel of $M(f)$ with the finite Dieudonn\'e module attached to $\ker f$; see \cite[Proposition~4.53(ii)]{CO09}.  A finite group scheme of rank $p^a$ corresponds to a finite Dieudonn\'e module of $W(k)$-length $a$.  Thus $\length_{W(k)}\coker M(f)=a.$ Put $r=\htg(G_1)=\htg(G_2)$.  By the elementary-divisor theorem, suitable integral bases give
\[
M(f)=\diag(p^{a_1},\ldots,p^{a_r}),\qquad a_i\ge0,
\]
and hence
\[
v_p(\det M(f))=\sum_i a_i=\length_{W(k)}\coker M(f)=a.
\]

For the final statement choose $e\ge0$ such that $f:=p^e\rho_z$ is an isogeny.  Since both groups have height $n$, $\htg(f)=h+ne.$ On rational Dieudonn\'e modules, $M(f)=p^eM(\rho_z).$ Therefore
\[
v_p(\det M(\rho_z))
=v_p(\det M(f))-ne=(h+ne)-ne=h.
\]
\end{proof}

We now define the model norm of $\det\alpha$.  The integral bundle $D$ is an integral model of $D_\eta$, while on the target we use the fixed lattice $N^\circ:=M(\mathbb H)^\vee\subset N.$ Let $\mathfrak U\subset\mathfrak M$ be a formal open on which $D$ is free, let $U=\mathfrak U_\eta$, and choose integral bases $d_1,\ldots,d_n$ of $D$ and $n_1,\ldots,n_n$ of $N^\circ$.  Put
\[
e_D=d_1\wedge\cdots\wedge d_n,
\qquad
e_N=n_1\wedge\cdots\wedge n_n.
\]
We declare these determinant generators to have norm $1$.  Thus, for $x\in U$ and $a,b\in\widehat{\kappa(x)}$,
\[
\|a e_D(x)\|=|a|,
\qquad
\|b e_N\|=|b|.
\]
If $A=(A_{ij})\in\GL_n(\mathcal O(U))$ is the matrix of $\alpha$ in these bases, then
\[
\det\alpha(e_D)=c e_N,
\qquad c=\det A\in\mathcal O(U)^\times,
\]
and we define
\[
|\det\alpha(x)|:=\|\det\alpha(x)\|=|c(x)|.
\]
Changing the integral determinant generators multiplies $c$ by an integral unit of absolute value $1$, so these local definitions glue to a global continuous model norm.

\begin{proposition}\label{prop:detnorm}
On the component $X=M_{U_0}^{(h)}$,
\[
|\det\alpha(x)|=\eta_h=p^{-h}\qquad(x\in X).
\tag{2.4}
\]
Consequently $\det D_\eta$ admits a global analytic trivialization of model norm $1$.
\end{proposition}

\begin{proof}
Let $z$ be a closed point of the special fibre.  Since $k$ is algebraically closed, $z$ is $k$-rational. 
The complete local deformation ring of a $p$-divisible group of height $n$ and dimension $n-2$ is $W(k)[[u_1,\ldots,u_{2(n-2)}]]$.  
Its generic fibre is the open unit $2(n-2)$-polydisc.

Choose integral bases of $D$ and $N^\circ$ on this formal neighbourhood. The crystalline framing is represented by an invertible matrix $A(u)$ of analytic functions, so $c(u)=\det A(u)$ is an analytic unit.  By \cref{lem:unitnorm}, $|c(x)|=|c(0)|$ throughout the 
deformation polydisc.

Under the chosen formal coordinates, $u_1=\cdots=u_{2(n-2)}=0$ determines a $W(k)$-valued lift $\widetilde z$ of $z$. Put
$R_j=W_j(k)=W(k)/p^j, \; \mathcal H_j=\mathcal H_{\widetilde z}\otimes_{W(k)}R_j.$ We use the covariant Dieudonn\'e crystal 
on the full crystalline site: the kernel of a divided-power thickening must be topologically nilpotent, but its divided-power 
structure need not be.  This includes the canonical thickening $\mathbf Z_2\twoheadrightarrow\mathbf F_2$; see \cite[\S3.2, 
immediately before Definition~3.2.3]{SW13}.  Thus $R_j\twoheadrightarrow k$ is allowed for every prime $p$.  Crystalline base 
change gives canonical integral identifications
\[
M(\mathcal H_z)(R_j\to k)
\simeq
M(\mathcal H_j)(R_j\xrightarrow{\id}R_j)
\simeq
D_{\mathcal H,\widetilde z}\otimes_{W(k)}R_j.
\]
They are compatible with $R_{j+1}\to R_j$, and hence
\[
\varprojlim_jM(\mathcal H_z)(R_j\to k)
\simeq
D_{\mathcal H,\widetilde z}.
\tag{2.5}
\]
Similarly,
\[
M(\mathbb H)\simeq\varprojlim_jM(\mathbb H)(R_j\to k).
\tag{2.6}
\]
Choose $e\ge0$ so that $f=p^e\rho_z$ is an isogeny.  The compatible crystalline maps induced by $f$ have inverse limit
\[
p^e r_\rho(0):M(\mathbb H)\longrightarrow D_{\mathcal H,\widetilde z}.
\]
Thus the source and target lattices are precisely those used in \cref{lem:heightdet}, and $v_p(\det r_\rho(0))=h.$ Since $\alpha=r_\rho^\vee$, its matrix in dual bases is the transpose of the matrix of $r_\rho$, so $v_p(c(0))=v_p(\det\alpha(0))=h.$ Because $|p|=p^{-1}$, $|c(0)|=p^{-h}.$ Therefore $|\det\alpha|=p^{-h}$ on the whole deformation polydisc.

Every classical point of $X$ lies in the tube of a closed special-fibre point.  Classical points are dense in strict affinoid charts, and the model norm is continuous.  Hence the equality holds at every analytic point of $X$.

Finally choose a reference volume form $\varepsilon_N\in\det N$ of norm $1$ and set $s=(\det\alpha)^{-1}(\varepsilon_N)\in\Gamma(X,\det D_\eta).$ Then $\|s\|=p^h$.  Thus
\[
\delta:=p^hs
\tag{2.7}
\]
is a nowhere-vanishing global analytic section of $\det D_\eta$ with $\|\delta\|=1$.
\end{proof}

\begin{remark}
The proposition asserts constancy of the norm, not of the analytic function itself.  For example, $1+T$ is nonconstant on the open unit disc but satisfies $|1+T|=1$ there.
\end{remark}

\subsection{Integral generators and the function map}\label{subsec:functionmap}
The Pl\"ucker coordinates $\ell_a$ are sections of $\lambda^{-1}$ rather than ordinary functions.  We first produce finitely many bounded sections of a positive power of $\lambda$, and then combine them with the $\ell_a$ to obtain global analytic functions.

\begin{proposition}\label{prop:generators}
There exist an even integer $w\ge2$ and finitely many sections $\sigma_i\in\Gamma(X,\lambda^w)$ such that
\[
\max_i\|\sigma_i(x)\|=1\qquad(x\in X).
\tag{2.8}
\]
The construction works for every prime $p$.
\end{proposition}

\begin{proof}
Theorem~\ref{thm:PEL} constructs an integral formal morphism $\kappa:\mathfrak M\longrightarrow\widehat{\mathscr S}$ to a split unitary PEL moduli space whose relevant abelian $p$-divisible group pulls back to $\mathcal H\times\mathcal H^\vee$.  Let $\omega_{\mathrm{ab}}$ denote the determinant of the Hodge bundle of the universal abelian scheme.  The Hodge sequences give
\[
\kappa^*\omega_{\mathrm{ab}}
=\det\bigl(L\oplus(D/L)^\vee\bigr)
=\lambda^2(\det D)^{-1}.
\tag{2.9}
\]
At neat level, $\omega_{\mathrm{ab}}$ extends to an ample invertible sheaf on the integral minimal compactification \cite[Theorem~7.2.4.1]{Lan21}.  Hence for some integer $m\ge1$ the power $\omega_{\mathrm{ab}}^m$ is generated by finitely many integral global sections.  Pulling them back gives integral sections $\tau_i$ generating $\lambda^{2m}(\det D)^{-m}$ on the formal model.

On the fixed-height generic fibre, let $\delta$ be the norm-one determinant trivialization of \cref{prop:detnorm}.  Define $\sigma_i:=\tau_i\delta^m\in\Gamma(X,\lambda^{2m}), \; w:=2m.$ Locally write $\tau_i=f_i e$ in an integral frame $e$ of the integral line bundle.  Since the $\tau_i$ generate, the $f_i$ generate the unit ideal.  Thus there exist integral $g_i$ with $\sum_i g_if_i=1$.  At every analytic point, integrality gives $|f_i|,|g_i|\le1$, so
\[
1=\left|\sum_i g_i(x)f_i(x)\right|\le\max_i|f_i(x)|\le1.
\]
Hence $\max_i\|\tau_i(x)\|=1$.  Since $\|\delta\|=1$, the same is true for the $\sigma_i$.
\end{proof}

For every triple $(i,a,b)$ define
\[
F_{iab}:=\sigma_i\ell_a^{w-1}\ell_b\in\Gamma(X,\mathcal O_X).
\tag{2.10}
\]
Let $d$ be the number of these functions and define
\[
\Psi=(F_{iab}):X\longrightarrow(\A_K^d)^{\an}.
\tag{2.11}\label{eq:Psi}
\]
Since $\max_i\|\sigma_i\|=1$ and $H_{\mathrm{Pl}}=\max_a\|\ell_a\|$,
\[
\max_{i,a,b}|F_{iab}(x)|=H_{\mathrm{Pl}}(x)^w.
\tag{2.12}\label{eq:functionheight}
\]
Indeed the upper bound is immediate, and equality is obtained by choosing $i$ with $\|\sigma_i(x)\|=1$ and $a$ with $\|\ell_a(x)\|=H_{\mathrm{Pl}}(x)$, then taking $b=a$.

\begin{lemma}\label{lem:functionmap}
The geometric fibres of $\Psi$ are discrete, and the differentials $dF_{iab}$ generate $\Omega^1_{X/K}$.
\end{lemma}

\begin{proof}
Fix $x\in X$.  Choose $i,a$ with $\sigma_i(x)\ell_a(x)\ne0$.  After shrinking, $F_{iaa}$ is invertible and
\[
\frac{F_{iab}}{F_{iaa}}=\frac{\ell_b}{\ell_a}.
\tag{2.13}\label{eq:plucker-ratio}
\]
The right-hand side consists of the affine Pl\"ucker coordinates on the chart $\{\ell_a\ne0\}$ of $\operatorname{Gr}(n-2,N)$.  Since the period map $\pi$ is \'etale, their differentials generate $\Omega^1_{X/K}$.  Differentiating \eqref{eq:plucker-ratio} gives
\[
d\!\left(\frac{\ell_b}{\ell_a}\right)
=
\frac{F_{iaa}\,dF_{iab}-F_{iab}\,dF_{iaa}}{F_{iaa}^2},
\]
so the $dF_{iab}$ generate $\Omega^1_{X/K}$.

On a fibre of $\Psi$, all $F_{iab}$ are fixed, so the ratios in \eqref{eq:plucker-ratio} are fixed.  Hence the fibre is locally contained in a fibre of the \'etale map $\pi$, and is therefore discrete.
\end{proof}

\section{Compact normalization in the universal cover}\label{sec:universalcover}
Fix a complete algebraically closed extension $C/K$, and let $\Perf_C$ denote the category of perfectoid spaces over $\Spa(C,\mathcal O_C)$.  Choose a lift $\mathbb H_0/W(k)$ of the framing group.  Let $\mathcal U_C$ be the base change to $C$ of the analytic universal cover associated with this lift.  It is preperfectoid, and we write $\widetilde{\mathbb H}:=\widehat{\mathcal U_C}$ for its strong completion in the sense of \cite[Proposition~2.3.6]{SW13}.  Then $\widetilde{\mathbb H}$ is perfectoid and, for every $T\in\Perf_C$,
\[
\Hom_C(T,\widetilde{\mathbb H})\xrightarrow{\sim}\Hom_C(T,\mathcal U_C).
\]
By \cite[Proposition~3.1.3 and Corollary~3.1.5]{SW13}, the resulting perfectoid universal cover is independent, up to canonical isomorphism, of the chosen lift.  We use the analogous convention after every complete algebraically closed extension $C'/C$.

Let $E\mathbb H_0$ be the universal vector extension of the chosen lift.  On integral test rings, the canonical map $s_{\mathbb H_0}:\widetilde{\mathbb H}_0\to E\mathbb H_0$ sends a compatible division sequence $(x_j)$ to $\lim_j p^j y_j$, where $y_j$ lifts $x_j$ to $E\mathbb H_0$.  The composite $\log_{E\mathbb H_0}\circ s_{\mathbb H_0}$, using the crystalline identification $\Lie E\mathbb H_0[1/p]\simeq M(\mathbb H)[1/p]$, is the quasi-logarithm of \cite[Definition~3.2.3]{SW13}.  Passing to generic fibres, base changing to $C$, and pulling back to the strong completion gives, for each affinoid perfectoid $T=\Spa(R,R^+)$, the $\Qp$-linear map
\[
q_T:\widetilde{\mathbb H}(T)\longrightarrow M(\mathbb H)\otimes_{W(k)}R.
\]
This construction is independent of the lift.  Its target is free of rank $n$ over $R$.  For $x\in\widetilde{\mathbb H}(T)$, write $x=(x_0,x_1,\ldots)$ relative to the chosen lift, with $[p]x_{i+1}=x_i$, and let $\pr_i(x)=x_i$.  The Hodge quotient $M(\mathbb H)\to\Lie\mathbb H_0$ sends $q_T(x)$ to $\log_{\mathbb H_0}(x_0)$ in the rank-$(n-2)$ module $\Lie\mathbb H_0\otimes_{W(k)}R$ \cite[Lemma~3.2.5 and Remark~3.2.6]{SW13}.

Let $X_T$ be the relative Fargues--Fontaine curve attached to $T$, and let $i_T:T\hookrightarrow X_T$ be the untilt divisor.  Write $\mathcal B_T$ for the pullback of $\mathcal B\simeq\mathcal O((n-2)/n)$ to $X_T$.

\begin{proposition}\label{prop:univcover-sections}
For every affinoid perfectoid $C$-space $T$, there is a canonical isomorphism, functorial in $T$,
\[
\widetilde{\mathbb H}(T)\xrightarrow{\sim}H^0(X_T,\mathcal B_T).
\tag{3.1}
\]
Under the canonical identification $i_T^*\mathcal B_T\simeq M(\mathbb H)\otimes_{W(k)}\mathcal O_T$, the quasi-logarithm is restriction to the untilt divisor:
\[
q_T(x)=i_T^*x.
\tag{3.2}
\]
\end{proposition}

\begin{proof}
Combining \cite[Theorems~15.2.3 and~15.2.5]{SW20} identifies the universal cover with global sections of the vector bundle attached to the relevant Frobenius module.  To align conventions, if $F$ is the Frobenius on the covariant Dieudonn\'e module used in \cite{SW13}, set $\Phi=p^{-1}F$.  Then $\Phi=1$ is equivalent to $F=p$, and more generally $\Phi=p^j$ is equivalent to $F=p^{j+1}$.  Thus the vector bundle obtained in the convention of \cite{SW20} is exactly $\mathcal E(\mathbb H)=\mathcal B$.

At a geometric rank-one point, the equality $q_T(x)=i_T^*x$ is \cite[Proposition~5.1.6(ii)]{SW13}.  For general $T=\Spa(R,R^+)$, express the difference in a basis of $M(\mathbb H)[1/p]$.  Its coefficients vanish at each rank-one point $y$ by pullback to $\Spa(C_y,\mathcal O_{C_y})$, where $C_y$ is a completed algebraic closure of $\widehat{\kappa(y)}$.  Since the perfectoid algebra $R$ is uniform, rank-one points detect its elements \cite[Proposition~4.2.5 and Theorem~5.2.1]{SW20}.  Hence the coefficients are zero, proving the equality on $T$.
\end{proof}

Fix the basis of $M(\mathbb H)[1/p]$ dual to the reference basis of $N$.  For an $n$-tuple $x=(x_1,\ldots,x_n)\in\widetilde{\mathbb H}(T)^n$, define
\[
Q_T(x):=(q_T(x_1),\ldots,q_T(x_n)).
\tag{3.3}
\]
The $n$ sections also define a vector-bundle morphism
\[
f_x:\mathcal O_{X_T}^{\,n}\longrightarrow\mathcal B_T,\qquad e_j\longmapsto x_j,
\]
and \cref{prop:univcover-sections} gives
\[
i_T^*f_x=Q_T(x).
\tag{3.4}
\]
For each pair of two-element subsets $I,J\subset\{1,\ldots,n\}$ let $m_{I,J}$ be the corresponding $2\times2$ minor of the universal quasi-logarithm matrix.  For an analytic point $y\in|\widetilde{\mathbb H}^{\,n}|$, define
\[
\Delta_2(y):=\max_{|I|=|J|=2}|m_{I,J}(y)|.
\tag{3.5}
\]
Thus $\Delta_2(y)>0$ precisely when the quasi-logarithm matrix has rank at least $2$ at $y$.

\subsection{Evaluation ranks}\label{subsec:evaluationranks}
Let $X_{\FF,C}$ denote the absolute Fargues--Fontaine curve, and let $i_\infty:\{\infty\}\hookrightarrow X_{\FF,C}$ be the inclusion of the untilt divisor determined by $C$, whose residue field is $C$.  The pullback $\mathcal B_\infty:=i_\infty^*\mathcal B$ is the fibre of $\mathcal B$ at $\infty$, an $n$-dimensional $C$-vector space.  For a global section $s$, pullback is evaluation: $i_\infty^*s=s(\infty)\in\mathcal B_\infty$.

\begin{lemma}\label{lem:evaluationrank}
A nonzero section of $\mathcal B$ has nonzero evaluation at $\infty$.  Any $\ell+1$ $\Qp$-linearly independent sections have evaluations spanning a space of dimension at least $2$.  The same statements hold after every complete algebraically closed extension of $C$.
\end{lemma}

\begin{proof}
Let $U\subset\mathcal B_\infty$ be a one-dimensional $C$-subspace and define the elementary modification
\[
\mathcal E_U:=\ker\!\left(\mathcal B\longrightarrow i_{\infty*}(\mathcal B_\infty/U)\right).
\]
Since $\rk\mathcal B=n$ and $\deg\mathcal B=n-2$,
\[
\deg\mathcal E_U=(n-2)-(n-1)=-1<0.
\tag{3.6}
\]
Suppose the positive Harder--Narasimhan part $\mathcal F\subset\mathcal E_U$ is nonzero.  Write $r=\rk\mathcal F$ and $e=\deg\mathcal F\ge1$.  Since $\deg\mathcal E_U<0$, the positive part cannot have full rank, so $r<n$.  Let $\mathcal F'\subset\mathcal B$ be its saturation.  Stability of $\mathcal B$ therefore gives $\deg\mathcal F'<\frac{(n-2)r}{n}$.  Moreover $\mathcal F=\mathcal F'\cap\mathcal E_U$: the two sheaves have the same generic subspace, and $\mathcal F$ is saturated in $\mathcal E_U$.  The map $\mathcal F'_\infty\to\mathcal B_\infty/U$ has rank at least $r-1$, so
\[
e\le \deg\mathcal F'-(r-1)
<\frac{(n-2)r}{n}-(r-1)
=1-\frac{2r}{n}<1,
\]
a contradiction.  Hence all Harder--Narasimhan slopes of $\mathcal E_U$ are nonpositive.

Let $r_0$ be the rank of its slope-zero part.  The negative total degree implies $r_0<n$, so any nonzero saturation of this part in $\mathcal B$ also has proper rank.  By the classification of vector bundles on the Fargues--Fontaine curve, the slope-zero part is $\mathcal O^{r_0}$ and the negative-slope part has no global sections.  Thus
\[
H^0(X_{\FF,C},\mathcal E_U)\simeq\Qp^{\,r_0}.
\tag{3.7}
\]
If $r_0>0$, saturating $\mathcal O^{r_0}$ inside $\mathcal B$ and repeating the previous estimate gives
\[
0<\frac{(n-2)r_0}{n}-(r_0-1),
\]
so $r_0<n/2$.  Therefore
\[
r_0\le\ell.
\tag{3.8}
\]
A section with zero evaluation lies in $H^0(\mathcal B(-1))=0$, since $\mathcal B(-1)$ has negative slope $-2/n$.  If $\ell+1$ independent sections had evaluation span of dimension at most one, all would lie in $H^0(\mathcal E_U)$ for some one-dimensional $U$, contradicting $r_0\le\ell$.  The argument is stable under complete algebraically closed scalar extension.
\end{proof}

\subsection{An invariant affinoid ball}\label{subsec:ball}
Choose formal coordinates $x_1,\ldots,x_{n-2}$ on the fixed lift $\mathbb H_0/W(k)$, and let $f_j=\pr_0^*x_j$ on the perfectoid universal cover.  Fix $m\ge1$ and put
\[
r=|p|^{1/m},
\qquad
B:=\{|f_j|\le r,\ 1\le j\le n-2\}\subset\widetilde{\mathbb H}.
\tag{3.9}
\]

\begin{lemma}\label{lem:ball}
The domain $B$ is an affinoid perfectoid $\Zp$-subgroup of $\widetilde{\mathbb H}$ and
\[
pB\subset\Int(B),
\qquad
\widetilde{\mathbb H}=\bigcup_{N\ge0}p^{-N}B.
\tag{3.10}
\]
There is an open Berkovich neighbourhood $V_0$ of $0$ with $V_0\subset pB$.
\end{lemma}

\begin{proof}
By \cite[Corollary~3.1.5]{SW13}, the formal universal cover has coordinates $T_1^{1/p^\infty},\ldots,T_{n-2}^{1/p^\infty}$.  The ideal $I=(p,f_1,\ldots,f_{n-2})$ is an ideal of definition by \cite[Lemma~3.1.6]{SW13}, while $J=(p,T_1,\ldots,T_{n-2})$ is the standard ideal of definition.  Hence $I$ and $J$ define the same adic topology, so choose $a\ge1$ with $J^a\subset I$.

For $y\in B$,
\[
|T_j(y)|^a\le\max\{|p|,|f_1(y)|,\ldots,|f_{n-2}(y)|\}\le r.
\]
Thus $B$ lies in a standard affinoid perfectoid ball.  Inside that ball the inequalities $|f_j|^m\le|p|$ are rational, proving that $B$ is affinoid perfectoid.

The formal group law, inversion, and multiplication by elements of $\Zp$ are given by power series with integral coefficients and zero constant term, and hence preserve the closed radius-$r$ ball.  Thus $B$ is a $\Zp$-subgroup.  The multiplication-by-$p$ series has the form
\[
[p](X)=pX+\text{terms of total degree at least }2,
\]
so on the radius-$r$ ball
\[
\|[p](X)\|\le\max(|p|r,r^2)<r.
\]
This gives $pB\subset\Int(B)$.  Iteration sends every point of the open formal group into $B$, and multiplication by $p$ is an automorphism of the universal cover, proving the exhaustion.

If $\pr_1$ denotes the next universal-cover projection, then $pB=\{|\pr_1^*x_j|^m\le|p|,\ 1\le j\le n-2\}$. The corresponding strict inequalities define the required open neighbourhood $V_0$.
\end{proof}

\subsection{Compact normalized frames}\label{subsec:normalizedframes}
Let $|B^n|_{\Berk}$ be the compact Hausdorff Berkovich spectrum of $B^n$.  Define
\[
K_B:=
\bigcap_{a\in\Zp^n\setminus p\Zp^n}
\left\{
(y_1,\ldots,y_n)\in|B^n|_{\Berk}:
\sum_{i=1}^n a_i y_i\notin V_0
\right\}.
\tag{3.11}
\]
For fixed $a$, the summation map is analytic and continuous, so every set in the intersection is closed.  Hence $K_B$ is compact.

At every geometric point of $K_B$, the $n$ entries are $\Qp$-linearly independent: a nonzero rational relation may be rescaled to a primitive vector in $\Zp^n\setminus p\Zp^n$, whose value would be $0\in V_0$.  The same statement holds at nonclassical Berkovich points after passage to a completed algebraic closure of the residue field.

By \cref{lem:evaluationrank}, the first column of $Q(y)$ is nonzero and the first $\ell+1$ have rank at least two.  Hence
\[
c(y):=
\max_{\substack{|I|=2\\2\le j\le\ell+1}}
|\det Q(y)_{I,\{1,j\}}|>0
\]
on $K_B$.  The function is continuous, so compactness yields
\[
c(y)\ge c_B>0\qquad(y\in K_B).
\tag{3.12}
\]

\begin{lemma}\label{lem:tatelattice}
Let $C'/C$ be complete and algebraically closed, and let $V\subset\widetilde{\mathbb H}(C')$ be the rational Tate space of an actual deformation, transported by its framing.  Then $L_0:=V\cap B(C')$ is a full $\Zp$-lattice in $V$, and every $\Zp$-basis of $L_0$ belongs to $K_B$.
\end{lemma}

\begin{proof}
Let $(\mathcal H,\rho)$ be the deformation over $\mathcal O_{C'}$.  Isogeny invariance and the infinitesimal lifting property of the universal cover give an isomorphism of perfectoid universal covers
\[
\widetilde\rho:\widetilde{\mathbb H}_{C'}\xrightarrow{\sim}\widetilde{\mathcal H}^{\,\mathrm{perf}}_\eta.
\]
The logarithm on the right has kernel the rational Tate space by \cite[Proposition~3.4.2(v)]{SW13}.  In a basis of $\Lie\mathcal H$, its restriction is $\ell_{\mathcal H,\rho}:B_{C'}\to(\A^{n-2}_{C'})^{\an}$.  Define the zero fibre in diamonds by $\mathscr Z_0:=B_{C'}^\diamond\times_{((\A^{n-2}_{C'})^{\an})^\diamond}\Spa(C',\mathcal O_{C'})^\diamond$, using $\ell_{\mathcal H,\rho}^\diamond$ and the origin.  Since $B_{C'}$ is affinoid perfectoid, \cite[Lemma~5.7]{HJ23}, applied to $\{0\}\hookrightarrow\A^{n-2}_{C'}$, represents $\mathscr Z_0$ by an affinoid perfectoid space $Z_0$ Zariski-closed in $B_{C'}$.  The full logarithmic kernel is the locally profinite perfectoid space attached to $V$ by \cite[Propositions~3.3.2 and~3.3.3]{SW13}.  Thus $L_0=V\cap B(C')$ is the continuous image of the quasi-compact underlying topological space $|Z_0|$ of $Z_0$ in the Hausdorff space $V$, and is compact.

It is a $\Zp$-submodule and contains the open neighbourhood $V_0\cap V$ of $0$.  Fix any lattice $T\simeq\Zp^n$ in $V$.  Since $L_0$ contains a neighbourhood of $0$, there is $a\ge0$ with $p^aT\subset L_0$.  Compactness makes $L_0$ bounded, so $L_0\subset p^{-b}T$ for some $b\ge0$.  Hence $L_0$ is a full lattice.

If $v_1,\ldots,v_n$ is a basis of $L_0$ and $a=(a_i)$ is primitive, then $v=\sum a_iv_i\notin pL_0$.  If $v\in pB$, then $p^{-1}v\in V\cap B(C')=L_0$, contradiction.  Since $V_0\subset pB$, one has $v\notin V_0$.  Therefore $(v_1,\ldots,v_n)\in K_B$.
\end{proof}

\section{The determinant on the rank locus}\label{sec:determinant}

\subsection{Relative Fargues--Fontaine conventions}
We retain $X_T$, $i_T$, and $\mathcal B_T$ from \cref{sec:universalcover}.  Functoriality of the Witt-vector construction and 
Frobenius quotient gives a natural morphism $\pi_T:X_T\longrightarrow X_{\FF,C}$ induced by 
$T^\flat\to\Spa(C^\flat,\mathcal O_{C^\flat})$; see \cite[Definition~II.1.15]{FS26}.  
The untilt divisor is the pullback of $\infty\subset X_{\FF,C}$ and $\mathcal B_T=\pi_T^*\mathcal B$.

Since $\mathcal B$ has rank $n$ and degree $n-2$, its determinant is the degree-$(n-2)$ line bundle 
$\det\mathcal B\simeq\mathcal O(n-2)$.  Put
\[
\mathcal L_{\det}:=(\det\mathcal B)\otimes\mathcal O(-(n-2))\simeq\mathcal O_{X_{\FF,C}},
\qquad
V_{\det}:=H^0(X_{\FF,C},\mathcal L_{\det}).
\]
Then $V_{\det}$ is one-dimensional over $\Qp$.  Fix a basis $\varepsilon\in V_{\det}$.  Evaluation gives a trivialization of $\mathcal L_{\det}$, and pullback gives compatible trivializations $\varepsilon_T$ of $\mathcal L_{\det,T}$ for every $T$.

Fix a compatible system of primitive $p$-power roots of unity in $C$ and let $E_\infty=\widehat{\Qp(\mu_{p^\infty})}$. Let $t$ be the corresponding cyclotomic period.  The relative Cartier sequence \cite[Proposition~II.2.3]{FS26} gives
\[
0\longrightarrow\mathcal O_{X_T}
\xrightarrow{\ t\ }
\mathcal O_{X_T}(1)
\longrightarrow i_{T*}\mathcal O_T
\longrightarrow0.
\tag{4.1}\label{eq:cartier}
\]
Using $\varepsilon_T$, \cite[Proposition~II.2.5(ii)]{FS26} identifies, functorially in $T$,
\[
H^0(X_T,\mathcal L_{\det,T})\simeq\underline{\Qp}(T).
\tag{4.2}\label{eq:BCdet}
\]
Here $\underline{\Qp}$ denotes the sheaf associated with $\Qp$ equipped with its usual $p$-adic topology: $\underline{\Qp}(T)=C^0(|T|,\Qp)$; see \cite[\S I.13]{FS26}.  Its sections are continuous and need not be locally constant.  In particular, for a general $T$ the right-hand side is not a single copy of $\Qp$.

\subsection{An analytic realization of the constant sheaf}
For $m\ge0$ put $P_m=p^{-m}\Zp$.  The affinoid perfectoid space
\[
\mathcal P_m
=
\Spa\bigl(C^0(P_m,C),C^0(P_m,\mathcal O_C)\bigr)
\]
represents the constant profinite sheaf $P_m$ on $\Perf_C$.  Although \cite[Example~11.12]{Sch18} is stated in 
characteristic $p$, the same construction applies over $C$: the supremum norm on $C^0(P_m,C)$ is power-multiplicative,
its power-bounded subring is $C^0(P_m,\mathcal O_C)$, and Frobenius on this subring modulo $p$ is surjective.  
Indeed, every continuous function modulo $p$ is constant on a finite clopen partition of $P_m$, and $p$th roots can 
be chosen on each part.  Writing $P_m=\varprojlim_i P_{m,i}$ with $P_{m,i}$ finite, the disjoint unions 
$P_{m,i}\times\Spa(C,\mathcal O_C)$ represent the finite constant sheaves.  Their inverse limit has the displayed 
algebra of continuous functions, since locally constant functions are uniformly dense, and hence represents $P_m$.

The inclusions $P_m\subset P_{m+1}$ are open and closed, so the $\mathcal P_m$ glue to a perfectoid space $\mathcal Q_C=\bigcup_{m\ge0}\mathcal P_m$ representing $\underline{\Qp}$.  On $\mathcal P_m$ the function $a\mapsto a\in C$ is analytic; the compatible functions define an adic morphism
\[
c:\mathcal Q_C\longrightarrow(\A_C^1)^{\an}.
\tag{4.3}\label{eq:Qc}
\]
Thus a natural $\underline{\Qp}$-valued invariant may be viewed, after composition with $c$, as an ordinary analytic function.

\subsection{Construction and moduli interpretation}
View $\A_C^{n^2}$ as the space of $n\times n$ matrices, and let $Y_{\le2}\subset\A_C^{n^2}$ be the algebraic determinantal variety defined by the $3\times3$ minors.  The entries of the universal quasi-logarithm matrix define an analytic map $Q:\widetilde{\mathbb H}^{\,n}\to(\A_C^{n^2})^{\an}$.  Form the diamond fibre product
\[
\mathscr Z:=(\widetilde{\mathbb H}^{\,n})^\diamond
\times_{((\A_C^{n^2})^{\an})^\diamond}
(Y_{\le2}^{\an})^\diamond.
\]
By \cite[Lemma~5.7]{HJ23}, $\mathscr Z$ is represented by a perfectoid space $Z$, with a Zariski-closed embedding $Z\hookrightarrow\widetilde{\mathbb H}^{\,n}$.  Its pullback to every affinoid perfectoid domain is affinoid perfectoid.  This is the meaning of the perfectoid closed subspace $Z$ used below.

\begin{lemma}\label{lem:DFF}
There is a functorial morphism $D_{\FF}:Z\longrightarrow\mathcal Q_C$ such that for every affinoid perfectoid $T\to Z$,
\[
\det f_x=t^{n-2}\varepsilon_T D_{\FF}(x).
\tag{4.4}\label{eq:DFFfactor}
\]
Moreover,
\[
D_{\FF}(xg)=\det(g)D_{\FF}(x)
\qquad(g\in\GL_n(\Qp)).
\tag{4.5}\label{eq:DFFcov}
\]
\end{lemma}

\begin{proof}
Let $T\to Z$ be affinoid perfectoid.  Since $i_T^*f_x=Q_T(x)$ has rank at most $2$, the restriction of $\det f_x$ to the untilt divisor vanishes.  Tensoring \eqref{eq:cartier} with the determinant line gives a unique first quotient by $t$.

At a geometric point $\Spa(C',\mathcal O_{C'})\to T$, the matrix of $f_x$ over $B_{\dR}^+(C')$ reduces modulo $t$ to a matrix of rank at most $2$.  Hence its determinant is divisible by $t^{n-2}$.  For each $1\le j<n-2$, once the $j$th quotient by $t$ has been constructed, it therefore restricts to zero at every geometric point.  After trivializing the relevant line bundle, this becomes a vanishing assertion for an analytic function on an affinoid perfectoid space.  Such an algebra is uniform and reduced, and nonzero functions are detected after passage to completed algebraic closures of rank-one residue fields.  Hence the quotient vanishes identically on the untilt divisor, and \eqref{eq:cartier} gives the next quotient.  After $n-2$ divisions, we obtain a unique section $s_T\in H^0(X_T,\mathcal L_{\det,T})$. The construction commutes with pullback.  Dividing by $\varepsilon_T$ and using \eqref{eq:BCdet}, the sections $s_T$ define a natural transformation $h_Z\to\underline{\Qp}=h_{\mathcal Q_C}$, hence a morphism $D_{\FF}:Z\to\mathcal Q_C$.  Equation \eqref{eq:DFFfactor} follows by construction.  Replacing $x$ by $xg$ replaces $f_x$ by $f_x\circ g$, and determinant multilinearity gives \eqref{eq:DFFcov}.
\end{proof}

Composing $D_{\FF}$ with $c$ in \eqref{eq:Qc}, we also regard $D_{\FF}$ as an analytic function on $Z$.

\begin{lemma}\label{lem:normshell}
Let $U=\Spa(R,R^+)\subset Z$ be affinoid perfectoid and put $\varphi=D_{\FF}|_U$.  Then
\[
U[\varphi\ne0]=\bigcup_{m\ge0}U(|p|^m\le|\varphi|).
\tag{4.6}\label{eq:nonzero-union}
\]
For every $\nu\in\mathbf Z$,
\[
U(|\varphi|=|p^\nu|)
=
U(|\varphi|\le|p^\nu|)\cap U(|p^\nu|\le|\varphi|)
\tag{4.7}
\]
is an affinoid perfectoid rational domain.  In particular $\varphi$ is locally an analytic unit on $U[\varphi\ne0]$.
\end{lemma}

\begin{proof}
The element $p$ is a topologically nilpotent unit in $R$.  If $|\varphi(x)|>0$, continuity of an adic valuation implies $|p(x)|^m\le|\varphi(x)|$ for some $m$, giving \eqref{eq:nonzero-union}.  The displayed inequalities define rational domains.  On $U(|p|^m\le|\varphi|)$, $\varphi^{-1}=p^{-m}(p^m/\varphi)$, so $\varphi$ is an analytic unit.  Rational domains in affinoid perfectoid spaces are affinoid perfectoid, and finite intersections remain rational.
\end{proof}

Put $Z^\times:=Z[D_{\FF}\ne0]$. \begin{proposition}\label{prop:ZequalsMinfty}
On $\Perf_C$, the functor represented by $Z^\times$ is naturally isomorphic to the infinite-level Rapoport--Zink functor $\M_{\infty,C}$.  The resulting maps to all finite levels are represented by actual adic morphisms.
\end{proposition}

\begin{proof}
First take a geometric point $x\in Z^\times(C',\mathcal O_{C'})$. Put $X'=X_{\FF,C'}$ and $\mathcal B'=\mathcal B_{\Spa(C',\mathcal O_{C'})}$, and write $i_{\infty'}$ for the untilt divisor on $X'$.  Since $D_{\FF}(x)\in\Qp^\times$, \eqref{eq:DFFfactor} implies $\operatorname{div}(\det f_x)=(n-2)\infty'$.  The Smith exponents of $f_x$ at $\infty'$ are nonnegative integers summing to $n-2$.  Since $x\in Z$, the reduction modulo $t$ has rank at most $2$, so at least $n-2$ exponents are positive.  They are therefore, up to permutation, $(\underbrace{1,\ldots,1}_{n-2},0,0)$. Hence $\rk Q(x)=2$ and
\[
0\longrightarrow\mathcal O_{X'}^n
\xrightarrow{f_x}\mathcal B'
\longrightarrow i_{\infty'*}W
\longrightarrow0,
\qquad \dim_{C'}W=n-2.
\tag{4.8}\label{eq:modification}
\]
By \cite[Corollary~6.3.10]{SW13}, this is a geometric point of $\M_\infty$.

Now let $U=\Spa(R,R^+)\subset Z^\times$ be affinoid perfectoid with $D_{\FF}$ invertible.  All $3\times3$ minors of $Q$ vanish in $R$.  Every point $y\in U$ has a rank-one generalization with the same support.  Passing from that generalization to a geometric point and using the preceding calculation shows that some $2\times2$ minor $\Delta$ satisfies $\Delta(y)\ne0$.  Thus the rational domains $U_{\Delta,m}:=U(|p|^m\le|\Delta|)$, as $\Delta$ ranges over these minors and $m\ge0$, cover $U$.  Quasi-compactness gives a finite subcover.  Each domain is affinoid perfectoid, and $\Delta$ is a unit there, with inverse $p^{-m}(p^m/\Delta)$.

On such a domain, permute rows and columns to write
\[
Q=\begin{pmatrix}A&B_1\\ B_2&B_3\end{pmatrix},
\qquad A\in\operatorname{Mat}_2(\mathcal O(U_{\Delta,m})),
\qquad \det A=\Delta.
\]
The $3\times3$ minors containing $A$ are $\Delta$ times the entries of $B_3-B_2A^{-1}B_1$.  Since they vanish and $\Delta$ is a unit, this Schur complement is zero.  Invertible row and column operations therefore identify $Q$ with $\diag(I_2,0_{n-2})$.  Its cokernel is free of rank $n-2$, giving condition~(i) of \cite[Definition~6.3.5]{SW13} on this rational domain.  Condition~(ii) holds at every geometric point by \eqref{eq:modification} and \cite[Corollary~6.3.10]{SW13}.  Hence \cite[Lemma~6.3.6]{SW13} gives an object of $\M_\infty$ on each domain.  These objects agree on overlaps under that natural isomorphism and glue on $U$.  The constructions for varying $U$ are compatible by the same functoriality.

Conversely, an object of $\M_\infty$ over an affinoid perfectoid $T$ satisfies the rank-two condition and hence induces $T\to Z$.  At every geometric point its Fargues--Fontaine modification has the form \eqref{eq:modification}, so $D_{\FF}$ is nonzero there.  Since every point of $T$ has a rank-one generalization with the same support, the pullback of $D_{\FF}$ is nonzero at every point of $T$.  The map therefore factors through $Z^\times$.  This proves the identification on perfectoid test spaces.

Composing with the natural finite-level maps $\M_\infty\to M_{U,C}$ gives compatible actual adic morphisms from $Z^\times$ to every finite level.
\end{proof}

Since $D_{\FF}$ is analytic and $Z\cap B^n$ is affinoid perfectoid, there exists an integer $A_B$ such that
\[
|D_{\FF}(y)|\le p^{A_B}\qquad(y\in Z\cap B^n).
\tag{4.9}\label{eq:DFFbound}
\]

\subsection{Descent of the determinant valuation}

\begin{lemma}\label{lem:nh}
For every height $h$, there is an integer $\nu_h$ such that every integral Tate frame $x$ above $M_{U_0,C}^{(h)}$ satisfies $v_pD_{\FF}(x)=\nu_h$. \end{lemma}

\begin{proof}
On $\M_{\infty,C}$, the determinant takes values in $\underline{\Qp^\times}$.  Since $v_p:\Qp^\times\to\mathbf Z$ is locally constant for the $p$-adic topology, $v_p\circ D_{\FF}$ is locally constant.  The map $\M_{\infty,C}\longrightarrow M_{U_0,C}$ is the pro-\'etale $U_0=\GL_n(\Zp)$-torsor of integral Tate frames.  By \eqref{eq:DFFcov}, changing the frame by $g\in U_0$ multiplies $D_{\FF}$ by $\det(g)\in\Zp^\times$, so the valuation is $U_0$-invariant and descends to a locally constant integer-valued function on $M_{U_0,C}$.

We show that a fixed-height component is geometrically connected.  For the datum $\mu=(1,1,0,\ldots,0),\; \nu_b=(2/n,\ldots,2/n)$, one has
\[
\mu-\nu_b=((n-2)/n,(n-2)/n,-2/n,\ldots,-2/n).
\]
Its coefficients in the simple-coroot basis are the first $n-1$ partial sums, namely $(n-2)/n$ for the first and $2(n-j)/n$ for $2\le j\le n-1$, all strictly positive.  Thus the datum is Hodge--Newton irreducible.  By the connected-component theorem \cite[\S5.1.3 and Theorem~5.1.10]{CKV15} (together with its corrigendum \cite{CKVcorr}), the hyperspecial geometric components over $\mathbf C_p$ form a torsor under
\[
G^{\mathrm{ab}}(\Qp)/\det(U_0)
=\Qp^\times/\Zp^\times\simeq\mathbf Z.
\]
Here the last identification is induced by $v_p$.  Take the framing lattice as origin and use the convention of transporting Dieudonn\'e lattices back along $M(\rho)^{-1}$.  The determinant valuation of the transported lattice is $-v_p(\det M(\rho))=-h$ by \cref{lem:heightdet}; this is its Kottwitz label.  Thus each fixed-height locus over $\mathbf C_p$ is one connected component.

Choose an embedding $\mathbf C_p\hookrightarrow C$ compatible with $K$.  Each such component is quasi-separated and has a $\mathbf C_p$-rational point, since $\mathbf C_p$ is algebraically closed.  By \cite[Theorem~3.2.1]{Con99}, it remains connected after any complete extension of $\mathbf C_p$.  Hence $M_{U_0,C}^{(h)}$ is geometrically connected.  The descended locally constant function is therefore constant on this locus; denote its value by $\nu_h$.
\end{proof}

\section{Uniform bounds for Tate lattices}\label{sec:uniform}
Let $V\subset\widetilde{\mathbb H}(C')$ be the rational Tate space of an actual deformation, and let $\Lambda\subset V$ be its transported integral Tate lattice.  For a $\Zp$-basis $x=(x_1,\ldots,x_n)$ of $\Lambda$, define $\Delta_2(\Lambda)=\Delta_2(x)$ and $v_pD_{\FF}(\Lambda)=v_pD_{\FF}(x)$. These are basis independent.  Indeed, replacing $x$ by $xg$ with $g\in\GL_n(\Zp)$ sends the vector of $2\times2$ minors through $\bigwedge^2g$; both $\bigwedge^2g$ and its inverse have integral coefficients, so the maximum norm is unchanged.  Also $D_{\FF}(xg)=\det(g)D_{\FF}(x)$ with $\det(g)\in\Zp^\times$.

\begin{theorem}\label{thm:uniformlattice}
Fix $\nu\in\mathbf Z$ and $R>0$.  Let $\Lambda$ be an actual integral Tate lattice over a complete algebraically closed extension $C'/C$ such that $v_pD_{\FF}(\Lambda)=\nu, \; \Delta_2(\Lambda)\le R$. Then
\[
\Lambda\subset p^{-N}B(C')
\]
with
\[
N=
\max\left\{
0,
\left\lceil(\ell+1)\log_p(R/c_B)+A_B+\nu\right\rceil
\right\}.
\tag{5.1}\label{eq:Nbound}
\]
The bound is independent of the period point, the deformation, and $C'$.
\end{theorem}

\begin{proof}
Put $V=\Lambda[1/p]$ and normalize it by the canonical lattice $L_0:=V\cap B(C')$ of \cref{lem:tatelattice}.  By the elementary-divisor theorem, choose a basis $y=(y_1,\ldots,y_n)$ of $L_0$ and integers $a_1\ge a_2\ge\cdots\ge a_n$ such that
\[
\Lambda=\bigoplus_{i=1}^n\Zp\,p^{-a_i}y_i.
\tag{5.2}
\]
No sign condition on the $a_i$ is imposed.

By \cref{lem:tatelattice}, $y\in K_B$.  Since $V$ is the rational Tate space of an actual deformation, $y$ differs from an actual Tate frame by an element of $\GL_n(\Qp)$.  The locus $Z$ is $\GL_n(\Qp)$-stable, so $y\in Z\cap B^n$.  The definition of $c_B$ then implies that among the column pairs $\{1,j\}$ with $2\le j\le\ell+1$, at least one has a $2\times2$ row minor of norm at least $c_B$.

Passing from $y_i$ to $p^{-a_i}y_i$ multiplies a minor using columns $1,j$ by the absolute-value factor $p^{a_1+a_j}$.  Since $j\le\ell+1$ and $a_j\ge a_{\ell+1}$,
\[
\Delta_2(\Lambda)\ge c_Bp^{a_1+a_{\ell+1}}.
\]
Using $\Delta_2(\Lambda)\le R$ gives
\[
a_1+a_{\ell+1}\le\log_p(R/c_B).
\tag{5.3}\label{eq:minor-sum}
\]

Put $S=\sum_i a_i$.  Determinant covariance gives
\[
D_{\FF}(p^{-a_1}y_1,\ldots,p^{-a_n}y_n)
=p^{-S}D_{\FF}(y).
\]
Taking valuations and using $v_pD_{\FF}(\Lambda)=\nu$,
\[
S=v_pD_{\FF}(y)-\nu.
\tag{5.4}
\]
Since $y\in Z\cap B^n$, \eqref{eq:DFFbound} gives $|D_{\FF}(y)|\le p^{A_B}$, equivalently
\[
S\ge-A_B-\nu.
\tag{5.5}\label{eq:S-lower}
\]
Finally,
\[
\begin{aligned}
(\ell+1)(a_1+a_{\ell+1})-S-a_1
&=\sum_{i=2}^{\ell}(a_1-a_i)
+\sum_{i=\ell+2}^{n}(a_{\ell+1}-a_i)\ge0.
\end{aligned}
\]
Combining \eqref{eq:minor-sum} and \eqref{eq:S-lower},
\[
a_1\le(\ell+1)\log_p(R/c_B)+A_B+\nu.
\tag{5.6}
\]
Thus every $a_i\le N$, so
\[
\Lambda\subset p^{-N}L_0\subset p^{-N}B(C').
\]
\end{proof}

The special column positions in \eqref{eq:minor-sum} are essential: the first evaluation is nonzero and the first $\ell+1$ have rank at least two, so one of the pairs $\{1,j\}$ with $2\le j\le\ell+1$ carries a uniformly nonzero minor.  These positions control $a_1+a_{\ell+1}$, and together with the total determinant exponent $S$ this bounds the largest elementary divisor $a_1$.

For later use, define
\[
Y_{N,\nu}:=(p^{-N}B)^n\cap Z\cap\{|D_{\FF}|=p^{-\nu}\}.
\tag{5.7}
\]
Multiplication by $p^{-N}$ is an automorphism of the universal cover, so $p^{-N}B$ is affinoid perfectoid.  The intersection $(p^{-N}B)^n\cap Z$ is therefore affinoid perfectoid.  By \cref{lem:normshell}, the norm shell $|D_{\FF}|=p^{-\nu}$ is the intersection of two rational domains, hence $Y_{N,\nu}$ is affinoid perfectoid.  Since $D_{\FF}$ does not vanish there, \cref{prop:ZequalsMinfty} supplies an actual adic morphism
\[
f_{N,\nu}:Y_{N,\nu}\longrightarrow M_{U_0,C}.
\tag{5.8}
\]
The domain is stable under integral changes of frame.  If $\nu=\nu_h$, then every geometric point of $M_{U_0,C}^{(h)}$ with $\Delta_2\le R$ lies in the image for the value of $N$ in \eqref{eq:Nbound}.

\section{Compact Hodge sublevels and Stein exhaustion}\label{sec:stein}

\subsection{The complementary-minor comparison}

\begin{lemma}\label{lem:complementary}
Let $\gamma$ be an integral Tate frame above a geometric point $x\in M_{U_0,C}^{(h)}$, and let $Q$ be the matrix of the quasi-logarithms of the transported Tate frame.  Then
\[
\Delta_2(Q)\le \eta_h^{-1}H_{\mathrm{Pl}}(x),\qquad \eta_h=p^{-h}.
\tag{6.1}\label{eq:complement-bound}
\]
The estimate is stable under complete algebraically closed scalar extension.
\end{lemma}

\begin{proof}
Work over the valuation ring of the geometric point.  Choose an integral basis $e_1,\ldots,e_n$ of $D$ whose first $n-2$ vectors form a basis of $L$.  Let $P$ be the matrix of the crystalline framing $\alpha$ in this basis and the fixed reference basis of $N$.  In the dual bases, $r_\rho=\alpha^\vee$ has matrix $P^T$, so $\beta=r_\rho^{-1}:D_{\mathcal H,\eta}\to N^\vee\otimes_K\mathcal O_X$ has matrix $P^{-T}:=(P^T)^{-1}$.  The last two vectors of the dual basis form a basis of $\Omega$.  Their images under $\beta$ form the $n\times2$ matrix
\[
J=(P^{-T})_{\{1,\ldots,n\},\{n-1,n\}}.
\tag{6.2}
\]
For every two-element row set $I$, Jacobi's identity gives
\[
\det J_I=\pm\frac{\det P_{I^c,\{1,\ldots,n-2\}}}{\det P}.
\tag{6.3}
\]
The numerators are the Pl\"ucker coordinates of the $(n-2)$-plane $\alpha(L)$.  By \cref{prop:detnorm},
\[
\max_{|I|=2}|\det J_I|=\eta_h^{-1}H_{\mathrm{Pl}}(x).
\tag{6.4}\label{eq:Jplucker}
\]

Let $C'$ be the complete algebraically closed field of the geometric point.  The integral Hodge--Tate map $\operatorname{HT}_{\mathcal H}:T_p\mathcal H\otimes_{\Zp}\mathcal O_{C'}\to\Omega=\omega_{\mathcal H^\vee}$ has a matrix $A_{\HT}\in\operatorname{Mat}_{2\times n}(\mathcal O_{C'})$ in the Tate frame $\gamma$ and the chosen basis of $\Omega$.  Its integrality follows from Cartier duality: a Tate vector induces a morphism $\mathcal H^\vee\to\mu_{p^\infty}$, and pulling back the invariant differential of $\mu_{p^\infty}$ gives an integral differential.  Consequently the $2\times2$ minors of $A_{\HT}$ have norm at most $1$.

Let $\iota_\Omega:\Omega\hookrightarrow D_{\mathcal H}$ be the Hodge inclusion, let $e_i^0$ be the standard basis of $\Zp^n$, and put $s_i=\widetilde\rho^{-1}(\gamma(e_i^0))$.  By \cite[Proposition~5.1.6(ii)--(iii) and the following Cartier-dual construction]{SW13}, the restriction of the quasi-logarithm to Tate vectors agrees, after extension of scalars to $C'$, with $\iota_\Omega\circ\operatorname{HT}_{\mathcal H}$.  Functoriality with respect to $\rho$ gives
\[
q_{\mathbb H}(s_i)
=\beta\bigl(q_{\mathcal H}(\gamma(e_i^0))\bigr)
=\beta\bigl(\iota_\Omega(\operatorname{HT}_{\mathcal H}(\gamma(e_i^0)))\bigr).
\]
Since $J$ represents $\beta\circ\iota_\Omega$, the preceding identity reads
\[
Q=JA_{\HT}.
\tag{6.5}
\]
Because the intermediate dimension is exactly $2$, each $2\times2$ minor factors as
\[
\det Q_{I,J'}=\det J_I\,\det(A_{\HT})_{\{1,2\},J'}.
\]
Taking absolute values and maxima and using \eqref{eq:Jplucker} proves \eqref{eq:complement-bound}.
\end{proof}

\subsection{Compact images at finite level}

\begin{lemma}\label{lem:compactimage}
Let $Y$ be a quasi-compact analytic adic space and let $T$ be a taut analytic adic space locally of finite type over a complete nontrivially valued field.  For an adic morphism $f:Y\to T$, its image in the associated Hausdorff Berkovich space $T_{\Berk}$ is compact.  If $Y$ is affinoid perfectoid, this image is also the image of its rank-one geometric points.
\end{lemma}

\begin{proof}
The separation map $\operatorname{sep}_T:|T|\to T_{\Berk}$ is continuous.  Hence $\operatorname{sep}_T\circ f$ has compact image because $|Y|$ is quasi-compact and $T_{\Berk}$ is Hausdorff; compare \cite[Chapter~8]{Hub96} and \cite[Definition~5.4, Remark~5.6, Theorem~7.9]{Hen16}.

Every point of an analytic adic space has a unique maximal generization, of rank one, and analytic morphisms preserve maximal generizations.  Passing from the completed residue field of such a point to a completed algebraic closure produces a geometric point with the same image under separation.  This proves the second assertion.
\end{proof}

\begin{proposition}\label{prop:Hodgecompact}
For every $R>0$, the Hodge sublevel $\{x\in M_{U_0,C}^{(h)}:H_{\mathrm{Pl}}(x)\le R\}$ is compact in the Berkovich topology.
\end{proposition}

\begin{proof}
Let $x$ be a Berkovich point of the sublevel and choose a completed algebraic closure $C'=\widehat{\overline{\widehat{\kappa(x)}}}$ of its completed residue field.  The resulting geometric point $x':\Spa(C',\mathcal O_{C'})\longrightarrow M_{U_0,C}^{(h)}$ has the same Berkovich image as $x$, and model norms are preserved by scalar extension.  Thus $H_{\mathrm{Pl}}(x')\le R$.

Choose an integral Tate frame above $x'$.  By \cref{lem:complementary}, $\Delta_2\le\eta_h^{-1}R$, while \cref{lem:nh} fixes the determinant valuation to $\nu_h$.  Theorem~\ref{thm:uniformlattice} gives an integer $N$ depending only on $R$ and $h$, so the frame defines a point of $Y_{N,\nu_h}$ mapping to $x'$.

Let $Y_{N,\nu_h}^{(h)}:=f_{N,\nu_h}^{-1}(M_{U_0,C}^{(h)})$. Height is locally constant, so this is open and closed in $Y_{N,\nu_h}$ and hence quasi-compact.  The preceding argument shows that its finite-level image contains every Berkovich point of the Hodge sublevel.

The Rapoport--Zink generic fibre is partially proper; this is used explicitly in \cite[Lemma~6.1.4]{SW13}.  Thus $M_{U_0,C}^{(h)}$ is taut by \cite[Remark~A.4]{Zav25}.  It is locally of finite type over $\Spa(C,\mathcal O_C)$, so \cite[Lemma~A.8]{Zav25} shows that its associated Berkovich space is without boundary over $C$.  By \cref{lem:compactimage}, the image of $Y_{N,\nu_h}^{(h)}$ is compact.  Since $H_{\mathrm{Pl}}$ is continuous, the sublevel is closed in this compact image and hence compact.
\end{proof}

\subsection{The finite function map}

\begin{lemma}\label{lem:finitecriterion}
Let $F$ be a complete nontrivially valued field and let $f:U\to V$ be a morphism of separated strictly $F$-analytic Berkovich spaces.  Suppose that
\begin{enumerate}[label=\textup{(\roman*)}]
\item the inverse image of every compact subset of $V$ is compact;
\item $f$ is without boundary;
\item every geometric fibre has dimension zero.
\end{enumerate}
Then $f$ is finite.  If $U$ is without boundary over $F$, condition \textup{(ii)} is automatic.
\end{lemma}

\begin{proof}
Condition \textup{(i)} makes $f$ topologically proper.  Together with the boundaryless condition this gives properness in the sense of Berkovich; see \cite[Lemma~3.9 and Remark~3.10]{CT21} and \cite[\S1.5]{Ber93}.  A boundaryless morphism with zero-dimensional fibres is quasi-finite, and a proper quasi-finite analytic morphism is finite.

For the last assertion, factor $f$ as its graph followed by projection:
\[
U\xrightarrow{\Gamma_f}U\times_FV\xrightarrow{\pr_2}V.
\]
Since $V$ is separated, the graph is a closed immersion and hence without boundary.  The projection is the base change of $U\to\M(F)$, which is without boundary by hypothesis.  The property is stable under base change and composition.
\end{proof}

\begin{theorem}\label{thm:functionfinite}
On every fixed-height component $X=M_{U_0}^{(h)}$, the function map $\Psi:X\longrightarrow(\A_K^d)^{\an}$ of \eqref{eq:Psi} is finite.
\end{theorem}

\begin{proof}
First base change to $C$.  Let $K_0\subset((\A_C^d)^{\an})_{\Berk}$ be compact.  The coordinate functions are uniformly bounded on $K_0$.  By \eqref{eq:functionheight}, $\Psi_C^{-1}(K_0)$ is contained in a Hodge sublevel, hence in a compact set by \cref{prop:Hodgecompact}.  Since $K_0$ is closed in the Hausdorff Berkovich affine space, $\Psi_C^{-1}(K_0)$ is compact.

$X_C$ is partially proper over $C$, hence taut by \cite[Remark~A.4]{Zav25}; since it is locally of finite type, its associated Berkovich space is without boundary over $C$ by \cite[Lemma~A.8]{Zav25}.  The target is separated.  By \cref{lem:functionmap}, all geometric fibres are zero-dimensional.  Lemma~\ref{lem:finitecriterion} therefore shows that $\Psi_C$ is finite.

All coordinate functions $F_{iab}$ are defined over $K$.  Finiteness descends under arbitrary complete extensions of the ground field by \cite[Theorem~11.5]{CT21}.  Hence $\Psi$ is finite over $K$.
\end{proof}

\begin{proof}[Proof of \cref{thm:main}]
First work on a fixed-height hyperspecial component $X=M_{U_0}^{(h)}$.  For $m\ge0$ put
\[
X_m:=\left\{x\in X:\max_{i,a,b}|F_{iab}(x)|\le p^m\right\}.
\tag{6.6}
\]
By \cref{thm:functionfinite}, $X_m$ is the inverse image of a closed polydisc under a finite morphism and is therefore affinoid.  The domains cover $X$ and satisfy $X_m\subset\Int(X_{m+1})$. They form an admissible covering: on a quasi-compact analytic domain the finitely many $F_{iab}$ are uniformly bounded, so the domain lies in some $X_m$.

Inside $X_{m+1}$, the domain $X_m$ is given by the Weierstrass inequalities $|p^mF_{iab}|\le1$. Its affinoid algebra is the corresponding completed Weierstrass algebra.  Polynomials in the auxiliary variables are dense, and the auxiliary variables equal the functions $p^mF_{iab}$.  Thus $\mathcal O(X_{m+1})\longrightarrow\mathcal O(X_m)$ has dense image.  This gives the desired Stein exhaustion of each height component.

Enumerate the height components as $M_{U_0}^{(h_j)}$, $j\ge0$, and denote their preceding exhaustions by $X_{j,m}$, $m\ge0$.  Define the diagonal sequence
\[
E_m:=\coprod_{j=0}^{m}X_{j,m}\qquad(m\ge0).
\]
Each $E_m$ is affinoid, with $\mathcal O(E_m)=\prod_{j=0}^{m}\mathcal O(X_{j,m})$, and $E_m\subset\Int(E_{m+1})$.  Every quasi-compact analytic domain meets only finitely many height components.  Its intersection with each such component is quasi-compact and lies in some $X_{j,m}$, so the entire domain lies in $E_m$ for sufficiently large $m$.  Thus the $E_m$ form an admissible exhaustion of $M_{U_0}$.

To describe $E_m$ as a Weierstrass domain of $E_{m+1}$, extend each defining function $p^mF_{iab}$ on an old component $X_{j,m+1}$, $j\le m$, by zero to the other components.  Impose their Weierstrass inequalities together with $|p^{-1}e_{m+1}|\le1$, where $e_{m+1}$ is the idempotent equal to $1$ on $X_{m+1,m+1}$ and $0$ elsewhere.  The last inequality excludes the new component, and the remaining inequalities cut out $X_{j,m}$ in every old component.  The restriction map $\mathcal O(E_{m+1})\to\mathcal O(E_m)$ is the finite product of the dense component restriction maps, with the new factor discarded, and therefore has dense image.

We now pass to arbitrary finite level in the original local Shimura tower, whose hyperspecial level is the working space of \cref{subsec:framing}.  Let $U\subset\GL_n(\Qp)$ be compact open.  Every compact subgroup stabilizes a $\Zp$-lattice, so choose $g\in\GL_n(\Qp)$ with $U':=g^{-1}Ug\subset U_0$.  The standard right Hecke isomorphism $M_U\xrightarrow{\sim}M_{U'}$, induced by $x\mapsto xg$, sends height $h$ to $h-v_p(\det g)$: the corresponding quasi-isogeny has height $-v_p(\det g)$ in our convention $\rho:\mathbb H\dashrightarrow\mathcal H$.  Thus it only translates the height labels by a fixed integer.

Since $U'$ is open, it has finite index in $U_0$, and the transition morphism $M_{U'}\longrightarrow M_{U_0}$ in the original tower is finite \'etale.  The inverse images of the $E_m$ are affinoid, form an admissible exhaustion, and satisfy the same interior inclusions.  Pulling back the defining inequalities makes each stage a Weierstrass domain in the next, so the restriction maps have dense image by the same polynomial approximation argument.  Transporting this exhaustion through the Hecke isomorphism gives the result for $M_U$.  The same construction applied to each height component proves the stated componentwise exhaustion at every finite level.
\end{proof}

\begin{appendices}

\section{An integral PEL realization of the formal moduli space}\label{app:PEL}
We construct the integral morphism used in \cref{prop:generators}.  Throughout this appendix, as in the body of the paper, 
$k=\overline{\mathbf F}_p$, and $\mathfrak M$ denotes the formal Rapoport--Zink space for the framing group $\mathbb H$ of height $n$ and dimension $n-2$ fixed in \cref{subsec:framing}.  The construction includes $p=2$ and is uniform in quasi-isogeny height.  The assumption that $k$ is algebraically closed is used to realize the framing group by a global PEL point.

\begin{theorem}\label{thm:PEL}
Let $p$ be any prime and $k=\overline{\mathbf F}_p$.  There exist a unitary PEL datum of signature $(n-2,2)$ which is split at $p$, a neat prime-to-$p$ level, and a smooth integral PEL moduli scheme $\mathscr S/W(k)$ with universal abelian scheme $\mathcal A$, together with a morphism of formal schemes
\[
\kappa:\mathfrak M\longrightarrow\widehat{\mathscr S}.
\tag{A.1}
\]
Here $\widehat{\mathscr S}$ denotes the $p$-adic completion of $\mathscr S$.
For every $p$-nilpotent $W(k)$-scheme $S$ and every $(\mathcal H_S,\rho)\in\mathfrak M(S)$, the associated abelian scheme $A_S/S$ satisfies
\[
A_S[p^\infty]\simeq\mathcal H_S\times\mathcal H_S^\vee
\tag{A.2}
\]
compatibly with the split endomorphism action and the standard hyperbolic principal quasi-polarization.  The polarization of $A_S$ has fixed degree prime to $p$, the prime-to-$p$ level is integral, and all constructions commute with arbitrary base change and with isomorphisms in $\mathfrak M(S)$.
\end{theorem}

\subsection{A global datum with good reduction at $p$}
Put
\[
\vartheta=\sqrt{1-4p},\qquad E=\mathbf Q(\vartheta),\qquad \theta=\frac{1+\vartheta}{2},\qquad \mathcal O=\mathcal O_E.
\]
The element $\theta$ satisfies $T^2-T+p=0$, whose reduction modulo $p$ is $T(T-1)$ with distinct roots, also for $p=2$.  Hence $p$ splits in $E$.  Moreover,
\[
\disc(\mathbf Z[\theta])=1-4p=[\mathcal O:\mathbf Z[\theta]]^2\disc(\mathcal O),
\]
so both the index and the discriminant of $\mathcal O$ are $p$-adic units.

Let $V_{\mathrm{PEL}}=E^n$ and equip it with the Hermitian form
\[
h_E(x,y)=\sum_{i=1}^{n-2}x_i\bar y_i-x_{n-1}\bar y_{n-1}-x_n\bar y_n.
\]
Define $\psi(x,y)=\Tr_{E/\mathbf Q}(\vartheta h_E(x,y)).$ Because $h_E(x,x)\in\mathbf Q$ and $\Tr_{E/\mathbf Q}(\vartheta)=0$, the form $\psi$ is alternating.  Take the lattice $\Lambda_{\mathrm{PEL}}=\mathcal O^n$.  Since $p$ is split and prime to $\disc(\mathcal O)$ and $\vartheta$ is a $p$-adic unit, the trace pairing is perfect at $p$; hence $\psi$ is perfect on $\Lambda_{\mathrm{PEL}}\otimes\Zp$.

At the complex embedding for which $\vartheta$ is positive imaginary, choose the complex structure acting by $i$ on the first $n-2$ coordinates and by $-i$ on the last two.  Then $\psi(x,Jx)>0$ for $x\ne0$, and the resulting PEL datum has signature $(n-2,2)$.  Since $p$ splits in $E$, $G_{\Qp}\simeq\mathrm{GL}_{n,\Qp}\times\Gm.$ The reflex field is $E$.  Choose the place above $p$ corresponding to an embedding $\tau:\mathcal O_E\hookrightarrow\Zp$, and put $\bar\tau(a):=\tau(\bar a)$.  The determinant condition is
\[
\det(T-\iota(a)\mid\Lie A)
=(T-\tau(a))^{n-2}(T-\bar\tau(a))^2
\qquad(a\in\mathcal O_E).
\tag{A.3}\label{eq:detcondition}
\]

Write $\operatorname{Disc}=\disc(\mathcal O_E)$ and let $\Lambda_{\mathrm{PEL}}^\#$ be the dual lattice for $\psi$.  The datum is of type A, so $I_{\mathrm{bad}}=1$ by \cite[Definition~1.2.1.18]{Lan21}.  The discriminant computation above gives $p\nmid\operatorname{Disc}$, and the perfectness of $\psi$ at $p$ gives $p\nmid[\Lambda_{\mathrm{PEL}}^\#:\Lambda_{\mathrm{PEL}}]$.  Choose $N_0\ge3$ prime to $p\operatorname{Disc}[\Lambda_{\mathrm{PEL}}^\#:\Lambda_{\mathrm{PEL}}]$ and a compact open prime-to-$p$ level contained in principal level $N_0$.  It is neat by \cite[Remark~1.4.1.9]{Lan21}.  Thus
\[
p\nmid N_0I_{\mathrm{bad}}\operatorname{Disc}[\Lambda_{\mathrm{PEL}}^\#:\Lambda_{\mathrm{PEL}}],
\]
which is precisely the good-prime condition of \cite[Definition~1.4.1.1]{Lan21}.  For $p=2$, the construction gives $E=\mathbf Q(\sqrt{-7})$, $\operatorname{Disc}=-7$, an odd lattice index, and odd $N_0$.  Moreover $\psi(x,x)=0$ holds integrally, so its reduction at $2$ is alternating.

By \cite[Theorem~1.4.1.11 and Corollary~7.2.3.10]{Lan21}, the resulting integral PEL moduli problem of neat level is represented by a smooth quasi-projective scheme with a universal abelian scheme.  Base change along $\mathcal O_E\otimes\mathbf Z_{(p)}\xrightarrow{\tau}\Zp\longrightarrow W(k)$ gives $\mathscr S/W(k)$.

\subsection{An exactly polarized framing point}
Set $\mathbb X:=\mathbb H\times\mathbb H^\vee.$ The two idempotents of $\mathcal O_E\otimes\Zp\simeq\Zp\times\Zp$ act on the two factors.  Give $\mathbb X$ the standard hyperbolic principal quasi-polarization, characterized under the bidual identification by the evaluation pairing between $\mathbb H$ and $\mathbb H^\vee$.  In the usual matrix convention it is
\[
\lambda_{\mathbb X}=\begin{pmatrix}0&1\\-1&0\end{pmatrix}.
\]
The two Lie ranks are $n-2$ and $2$, so the determinant condition is exactly \eqref{eq:detcondition}.

Thus $(\mathbb X,\iota,\lambda_{\mathbb X})$ satisfies \cite[Definition~1.11]{VW13} for the PEL datum just constructed.  The reflex place is split, so its residue field is $\mathbf F_p$.  Since $k=\overline{\mathbf F}_p$, the algebraic-closedness hypothesis of \cite[Theorem~10.2]{VW13} is satisfied.  That theorem gives a point over $k$ whose $p$-divisible group, with its endomorphism action and polarization class, is isomorphic to $\mathbb X$.  The datum has type A, with local type $(AL)$, and hence has no excluded type D factor.  The theorem applies to every prime; the type $(AL)$ deformation argument at $2$ is also discussed in \cite[Remark~9.4]{VW13}.

Applying \cite[Proposition~1.4.3.4]{Lan21} over $\Spec k$ gives an integral representative $(A_0,\lambda_0,\iota_0,\eta_0)\in\mathscr S(k)$, with an actual $\mathcal O_E$-action, a polarization of degree $d_0$ prime to $p$, and integral prime-to-$p$ level.  This comparison uses a prime-to-$p$ quasi-isogeny and therefore preserves the isomorphism class of the $p$-divisible group with its polarization class.

Initially one obtains an $\mathcal O_E$-linear isomorphism $\zeta_0:\mathbb X\xrightarrow{\sim}A_0[p^\infty]$ with $\zeta_0^*\lambda_0[p^\infty]=u\lambda_{\mathbb X},\; u\in\Zp^\times,$ because the local PEL polarization is recorded up to $\Zp^\times$-similitude.  Let $c=\diag(u^{-1},1)$ on $\mathbb H\times\mathbb H^\vee$.  For the hyperbolic form, a diagonal automorphism $\diag(a,b)$ has multiplier $ab$, so $c^*\lambda_{\mathbb X}=u^{-1}\lambda_{\mathbb X}.$ Replacing $\zeta_0$ by $\zeta=\zeta_0c$ gives the exact identity
\[
\zeta^*\lambda_0[p^\infty]=\lambda_{\mathbb X}.
\tag{A.4}\label{eq:exactpolar}
\]

\subsection{Clearing denominators and taking the quotient}
Let $S$ be $p$-nilpotent over $W(k)$, put $\bar S=S\times_{W(k)}k$, and take $(\mathcal H_S,\rho)\in\mathfrak M(S)$.  Set $\mathbb X_S=\mathcal H_S\times\mathcal H_S^\vee.$ The doubled quasi-isogeny
\[
\widetilde\rho
=\rho\times(\rho^\vee)^{-1}:
\mathbb X_{\bar S}\dashrightarrow\mathbb X_S|_{\bar S}
\]
has height zero and preserves the hyperbolic quasi-polarization.

Zariski locally on $\bar S$, choose $e\ge0$ so that
\[
a_e=p^e\widetilde\rho\,\zeta^{-1}:
A_{0,\bar S}[p^\infty]\longrightarrow\mathbb X_S|_{\bar S}
\]
is an isogeny.  Let $K_e=\ker(a_e)$.  It is finite locally free and stable under $\mathcal O_E$.  Form the quotient
\[
f_e:A_{0,\bar S}\longrightarrow A_e:=A_{0,\bar S}/K_e.
\tag{A.5}
\]
By the finite-flat quotient theorem for abelian schemes, applied after noetherian approximation, $A_e$ is an abelian scheme; see the quotient construction in \cite[proof of Theorem~1.2.1]{Kat81}.  The isogeny $a_e$ factors uniquely through an isomorphism $\epsilon_e:A_e[p^\infty]\xrightarrow{\sim}\mathbb X_S|_{\bar S}.$ The $\mathcal O_E$-action descends through $f_e$.

\begin{lemma}\label{lem:qe}
There is a unique polarization $q_e:A_e\longrightarrow A_e^\vee$ satisfying
\[
f_e^\vee q_e f_e=p^{2e}\lambda_0.
\tag{A.6}\label{eq:qe}
\]
It has degree $d_0$, is compatible with the descended $\mathcal O_E$-action under Rosati, and its $p$-divisible realization corresponds under $\epsilon_e$ to the standard hyperbolic principal quasi-polarization of $\mathbb X_S$.
\end{lemma}

\begin{proof}
In the category of abelian schemes up to $p$-isogeny define $q_e=(f_e^\vee)^{-1}p^{2e}\lambda_0f_e^{-1}.$ By \eqref{eq:exactpolar} and the definition of $a_e$, its rational $p$-divisible realization is exactly the integral hyperbolic principal quasi-polarization transported through $\epsilon_e$.

Choose $m\ge0$ such that $u_e=p^mq_e$ is an actual abelian-scheme homomorphism.  On $p$-divisible groups, $u_e$ equals $p^m$ times an integral homomorphism, and therefore kills $A_e[p^m]$.  Since $[p^m]:A_e\to A_e$ is the quotient by $A_e[p^m]$, the universal property of the quotient gives a unique actual homomorphism $q_e'$ with $u_e=q_e'\circ[p^m].$ In $\Hom(A_e,A_e^\vee)[1/p]$ one has $q_e'=q_e$.  Thus $q_e$ is integral.

Symmetry and Rosati compatibility follow from \eqref{eq:qe}, using that homomorphism groups of abelian schemes have no $p$-torsion.  For a symmetric homomorphism $q:A\to A^\vee$, set $\mathcal L(q):=(\id_A,q)^*\mathcal P_A$, where $\mathcal P_A$ is the Poincar\'e bundle.  Then $\lambda_{\mathcal L(q)}=q+q^\vee=2q$; see \cite[proof of Proposition~1.3.2.14]{Lan21}.  By \cite[Proposition~1.3.2.15]{Lan21}, $q$ is a polarization if and only if $\mathcal L(q)$ is relatively ample.  Poincar\'e functoriality \cite[Lemma~1.3.2.10]{Lan21} and \eqref{eq:qe} give
\[
f_e^*\mathcal L(q_e)
\simeq\mathcal L(f_e^\vee q_e f_e)
=\mathcal L(p^{2e}\lambda_0)
\simeq\mathcal L(\lambda_0)^{\otimes p^{2e}}.
\]
Since $\lambda_0$ is a polarization, this line bundle is relatively ample.  Relative ampleness descends along the finite locally free surjection $f_e$; see \cite[proof of Corollary~1.3.2.21]{Lan21}.  The criterion therefore shows that $q_e$ is a polarization.  It applies over arbitrary bases by the approximation argument in \cite[proof of Proposition~1.3.2.15]{Lan21} and requires no~inversion~of~$2$.

The doubled group has height $2n$ and $\widetilde\rho$ has height zero, hence $\deg f_e=p^{2ne}$.  Since $\dim A_0=n$,
\[
(\deg f_e)^2\deg q_e=\deg(p^{2e}\lambda_0)=p^{4ne}d_0,
\]
and therefore $\deg q_e=d_0$.  Uniqueness follows by cancelling $f_e$ in the $p$-isogeny category and then using torsion-freeness.
\end{proof}

Transport the chosen prime-to-$p$ level by
\[
\eta_e=p^{-e}f_e\eta_0.
\tag{A.7}
\]
The scalar $p^{-e}$ is a unit on every prime-to-$p$ Tate module, so $\eta_e$ is again integral.  Equation \eqref{eq:qe} shows that its polarization multiplier is $1$.  Consequently $(A_e,q_e,\iota_e,\eta_e)$ is an object of the fixed integral PEL moduli problem.

\subsection{Independence of the denominator and descent}

\begin{lemma}\label{lem:denominator}
The locally constructed polarized PEL objects, together with the identifications $\epsilon_e$, are canonically independent of the choice of clearing exponent $e$ and glue over $\bar S$.  The gluing is compatible with arbitrary base change.
\end{lemma}

\begin{proof}
Suppose $d\ge e$ are two allowed exponents.  Then $a_d=[p^{d-e}]a_e$, and under $\epsilon_e$ one has $K_d/K_e=A_e[p^{d-e}].$ Hence multiplication by $p^{d-e}$ induces a unique isomorphism $j_{d,e}:A_d\xrightarrow{\sim}A_e$ satisfying
\[
j_{d,e}f_d=[p^{d-e}]f_e.
\tag{A.8}
\]
It is $\mathcal O_E$-linear.  Pulling back \eqref{eq:qe} gives
\[
f_d^*j_{d,e}^*q_e=[p^{d-e}]^*f_e^*q_e=p^{2d}\lambda_0=f_d^*q_d,
\]
so $j_{d,e}^*q_e=q_d$.  Similarly $j_{d,e}\eta_d=\eta_e$.  The $p$-divisible-group identifications are compatible as well: using $\epsilon_e\circ f_e[p^\infty]=a_e$, we obtain
\[
(\epsilon_e\circ j_{d,e}[p^\infty])\circ f_d[p^\infty]
=[p^{d-e}]\circ a_e
=a_d
=\epsilon_d\circ f_d[p^\infty].
\]
Since the isogeny $f_d[p^\infty]$ is an epimorphism of fppf sheaves, cancellation gives $\epsilon_e\circ j_{d,e}[p^\infty]=\epsilon_d$.

For $r\ge d\ge e$, uniqueness gives the cocycle identity $j_{d,e}j_{r,d}=j_{r,e}.$ On overlaps choose a common larger exponent and use the resulting isomorphisms.  The cocycle condition gives canonical Zariski descent, including the identifications $\epsilon_e$.  Since schemes, morphisms, finite \'etale level structures, smoothness, properness, and group-scheme operations descend effectively, we obtain a global object $(\bar A,\bar q,\bar\iota,\bar\eta)$ over $\bar S$ together with
\[
\bar\epsilon:\bar A[p^\infty]\xrightarrow{\sim}\mathbb X_S|_{\bar S}.
\tag{A.9}
\]
All quotient constructions commute with base change, so the descent is functorial.
\end{proof}

\subsection{Serre--Tate lifting and the formal morphism}
The closed immersion $\bar S\hookrightarrow S$ is defined by the nilpotent ideal $p\mathcal O_S$.  Applying the Serre--Tate equivalence \cite[Theorem~1.2.1]{Kat81} on affine opens and gluing by uniqueness lifts $(\bar A,\bar\epsilon)$ to an abelian scheme $A/S$ together with
\[
\epsilon:A[p^\infty]\xrightarrow{\sim}\mathbb X_S.
\tag{A.10}
\]
The $\mathcal O_E$-action lifts uniquely by full faithfulness.  The polarization $\bar q$ and its prescribed $p$-divisible realization lift uniquely to a homomorphism $q:A\to A^\vee$.  Symmetry and Rosati compatibility lift as equalities of homomorphisms, and fibrewise positivity shows that $q$ is a polarization of degree $d_0$.  Prime-to-$p$ torsion is finite \'etale and unchanged by nilpotent thickening, so the integral level structure lifts uniquely as well.

Finally, $\Lie A\simeq\Lie\mathcal H_S\oplus\Lie\mathcal H_S^\vee$ has ranks $n-2$ and $2$ under the two idempotents, and $a\in\mathcal O_E$ acts by the corresponding scalars $\tau(a)$ and $\bar\tau(a)$.  Thus the determinant condition \eqref{eq:detcondition} holds over the possibly nonreduced test scheme.

The construction is functorial in $S$ and in morphisms of framed deformations.  We therefore obtain a natural transformation of represented functors, and formal Yoneda gives
\[
\kappa:\mathfrak M\longrightarrow\widehat{\mathscr S}.
\tag{A.11}
\]
This proves \cref{thm:PEL}.

\subsection{Integral Hodge generators}
Let $\omega_{\mathrm{ab}}=\det\omega_{\mathcal A/\mathscr S}.$ The integral $p$-divisible-group identification and the Hodge sequences give
\[
\kappa^*\omega_{\mathrm{ab}}
=\det(\omega_{\mathcal H}\oplus\omega_{\mathcal H^\vee})
=\lambda^2(\det D)^{-1}.
\tag{A.12}
\]
Indeed $\omega_{\mathcal H}=L$, $\omega_{\mathcal H^\vee}=(D/L)^\vee$, and $\lambda=\det L$.

By \cite[Theorem~7.2.4.1]{Lan21}, the chosen neat-level integral PEL scheme has a projective minimal compactification on which $\omega_{\mathrm{ab}}$ extends to an ample invertible sheaf.  Since the base is affine and noetherian, a sufficiently large fixed power is generated by finitely many integral global sections.  After base change to $W(k)$, restriction to the $p$-adic completion, and pullback by $\kappa$, we obtain finitely many integral generators
\[
\tau_i\in\Gamma\!\left(\mathfrak M,\lambda^{2a}(\det D)^{-a}\right)
\tag{A.13}
\]
for a single integer $a>0$, independent of height.

On a generic fibre of fixed height, let $\delta$ be the norm-one determinant trivialization of \cref{prop:detnorm}.  Then $\sigma_i:=\tau_i\delta^a\in\Gamma(X,\lambda^{2a}).$ The $\tau_i$ generate the integral line bundle, so in an integral local frame their coefficients generate the unit ideal.  The argument in the proof of \cref{prop:generators} gives
\[
\max_i\|\sigma_i(x)\|=1.
\]
Taking $w=2a$ proves \cref{prop:generators} for every prime, including $p=2$.

\end{appendices}

\end{document}